\documentclass{scrartcl}
\usepackage[utf8]{inputenc}
\usepackage[T1]{fontenc}
\usepackage{lmodern}
\usepackage{microtype}
\usepackage[english]{babel}
\usepackage{amsmath, amssymb, mathtools}
\usepackage[dvipsnames]{xcolor}
\usepackage{xparse}
\usepackage{caption}
\usepackage{subcaption}
\usepackage{enumitem}

\usepackage{tikz}
\usepackage[colorinlistoftodos]{todonotes}
\usetikzlibrary{cd,calc,quotes}
\usepackage{forest}

\usetikzlibrary{decorations.pathmorphing}

\usepackage[autostyle]{csquotes}
\usepackage[style=alphabetic, sorting=nyt, maxnames=10, maxalphanames=5]{biblatex}
\DeclareDelimFormat[textcite]{multinamedelim}{\textendash}
\DeclareDelimAlias[textcite]{finalnamedelim}{multinamedelim}

\usepackage[numbered]{bookmark}
\usepackage{hyperref}
\definecolor{link}{HTML}{C10000}
\definecolor{cite}{HTML}{006F00}
\definecolor{url}{HTML}{3C3CFF}
\hypersetup{pdfencoding=auto, unicode, colorlinks, citecolor=cite, urlcolor=url, linkcolor=link}
\usepackage{amsthm, thmtools}
\usepackage{mathrsfs}

\newcommand{\ourDR}[5]{
  \begin{tikzcd}[ampersand replacement=\&]
    #1 \arrow[r,shift left = .5 ex,"#3"] \& #2 \arrow[l, shift left = .5 ex,"#4"] \arrow[loop right, distance = 4em,start anchor = {[yshift = 1ex]east},end anchor = {[yshift=-1ex]east}]{}{#5}
  \end{tikzcd}
}

\forestset{
operad/.style={
font={\footnotesize},
edge={thick},
for tree={grow'=90, l sep=.9cm, s sep = .9cm},
for level>=1{circle, fill=black, inner sep=.125em},
for leaves={fill=none, inner sep=.125em, font={\small}}
},
smalloperad/.style={
    operad,
    for tree={l = .5cm, l sep = .5cm, s sep = .33cm}
  },
lb/.style={
    edge label={
        node [midway, fill=white, inner sep = 1pt, font=\scriptsize, align=center] {#1}
      }
  },
mod/.style={edge={red, thick, densely dashed}},
mod2/.style={edge={blue, thick, double, densely dashed}},
}

\makeatletter
\newcommand{\subalign}[1]{%
  \vcenter{%
    \Let@ \restore@math@cr \default@tag
    \baselineskip\fontdimen10 \scriptfont\tw@
    \advance\baselineskip\fontdimen12 \scriptfont\tw@
    \lineskip\thr@@\fontdimen8 \scriptfont\thr@@
    \lineskiplimit\lineskip
    \ialign{\hfil$\m@th\scriptstyle##$&$\m@th\scriptstyle{}##$\hfil\crcr
      #1\crcr
    }%
  }%
}
\makeatother

\NewDocumentCommand{\eugene}{sO{}m}{\todo[\IfBooleanT{#1}{inline}, #2, color=Plum, textcolor=white]{E: #3}}

\newcommand{\antishriek}{\text{\raisebox{\depth}{\textexclamdown}}}

\makeatletter
\def\font@loop#1{
  \ifx\relax#1
  \else
    \expandafter\def\csname c#1\endcsname{\mathcal{#1}}
    \expandafter\def\csname s#1\endcsname{\mathsf{#1}}
    \expandafter\def\csname b#1\endcsname{\mathbf{#1}}
    \expandafter\def\csname f#1\endcsname{\mathfrak{#1}}
    \expandafter\def\csname #1#1\endcsname{\mathbb{#1}}
    \expandafter\font@loop%
  \fi}
\font@loop ABCDEFGHIJKLMNOPQRSTUVWXYZ\relax
\makeatother
\def\fg{\mathfrak g}

\DeclareMathOperator{\sgn}{sgn}
\DeclareMathOperator{\Sym}{Sym}
\DeclareMathOperator{\Tw}{Tw}
\DeclareMathOperator{\Hom}{Hom}
\DeclareMathOperator{\Alg}{Alg}
\DeclareMathOperator{\Op}{Op}
\DeclareMathOperator{\CoOp}{CoOp}
\DeclareMathOperator{\Sh}{Sh}
\DeclareMathOperator{\Com}{Com}
\DeclareMathOperator{\Lie}{Lie}
\DeclareMathOperator{\As}{As}
\DeclareMathOperator{\Disj}{Disj}
\DeclareMathOperator{\hoDisj}{hoDisj}
\DeclareMathOperator{\PCosh}{Opens}

\DeclareMathOperator{\Tens}{Tens}

\DeclareMathOperator{\Ch}{Ch}

\newcommand{\cocirc}{\mathbin{\underline{\circ}}}
\DeclareMathOperator{\id}{id}
\NewDocumentCommand{\bin}{O{U}O{V}}{\ensuremath{\mu_{#1,#2}}}
\NewDocumentCommand{\ext}{O{U}O{V}}{\ensuremath{\iota_{#1}^{#2}}}
\NewDocumentCommand{\binshriek}{O{U}O{V}}{\ensuremath{\mu^!_{#1,#2}}}
\NewDocumentCommand{\extshriek}{O{U}O{V}}{\ensuremath{\iota_{#1}^{!,#2}}}
\NewDocumentCommand{\binantishriek}{O{U}O{V}}{\ensuremath{\mu^\antishriek_{#1,#2}}}
\NewDocumentCommand{\extantishriek}{O{U}O{V}}{\ensuremath{\iota_{#1}^{\antishriek,#2}}}
\NewDocumentCommand{\ops}{O{\mathcal{U}}}{\ensuremath{\mu_{#1}}}

\NewDocumentCommand{\taui}{mO{U,V,W}}{\mathop{\tau_{\mathrm{#1}}}(#2)}

\declaretheoremstyle[
  spaceabove=7pt, spacebelow=7pt,
  headfont=\normalfont\bfseries,
  notefont=\mdseries, notebraces={()},
  bodyfont=\normalfont,
  postheadspace=5pt,
  headpunct = .
]{thm}

\declaretheoremstyle[
  spaceabove=7pt, spacebelow=7pt,
  headfont=\normalfont\bfseries,
  notefont=\mdseries, notebraces={()},
  bodyfont=\normalfont,
  postheadspace=10pt,
  headpunct = .
]{def}

\declaretheoremstyle[
  spaceabove=4pt, spacebelow=7pt,
  headfont=\itshape,
  postheadspace=5pt,
  headpunct = :,
  postheadspace = 3pt, qed = $\lozenge$
]{rem}

\declaretheorem[numbered = yes, parent = section, style = thm]{theorem}

\declaretheorem[style = thm, name = Theorem]{thmABC}

\declaretheorem[sibling = theorem, style = thm, name = Proposition/Definition]{propdef}
\declaretheorem[sibling = theorem, style = thm]{proposition}
\declaretheorem[sibling = theorem, style = thm]{lemma}

\numberwithin{equation}{section}

\declaretheorem[sibling = theorem, style = def]{definition}

\declaretheorem[sibling = theorem, style = rem]{remark}
\declaretheorem[numbered = no, style = rem, name = Remark]{uremark}

\declaretheorem[sibling = theorem, style = rem]{example}

\title{Homotopy Prefactorization Algebras}
\author{
  Najib Idrissi%
  \thanks{Université Paris Cité and Sorbonne Université, CNRS, IMJ-PRG, F-75013 Paris, France. \href{mailto:najib.idrissi-kaitouni@u-paris.fr}{najib.idrissi-kaitouni@u-paris.fr}}
  \and
  Eugene Rabinovich%
  \thanks{Succint, Inc. \href{mailto:erabinovich92@gmail.com}{erabinovich92@gmail.com}}
  }
  \date{June 2024}

\begin{document}

\maketitle

\begin{abstract}
  We apply the theory of operadic Koszul duality to provide a cofibrant resolution of the colored operad whose algebras are prefactorization algebras on a fixed space $M$.
  This allows us to describe a notion of prefactorization algebra up to homotopy as well as morphisms up to homotopy between such objects.
  We make explicit these notions for several special $M$, such as certain finite topological spaces, or the real line.

  ~

\end{abstract}

\tableofcontents

\section{Introduction}

Prefactorization algebras and factorization algebras---their cousins satisfying descent---are objects meant to model the structure present in the observables of quantum field theory (see~\autocite{CG1, CG2}).
A prefactorization algebra on a topological space $M$ is a precosheaf $\cF$ on $M$ equipped with extra structure maps that can ``glue'' elements of $\cF(U_1) ,\ldots , \cF(U_n)$ (for pairwise disjoint open subsets $U_i \subset V$) into a single element of $\cF (V)$ (see Definition~\ref{def pfa}).
(Pre-)factorization algebras with additional properties give rise to familiar algebraic objects, such as vertex algebras or $E_n$-algebras.
Some tools to study the homotopy theory of prefactorization algebras have been given by \textcite{CarmonaFloresMuro2021}, including a model structure on the category of prefactorization algebras on $M$ such that the bifibrant objects in this category are (a subclass of the) locally constant factorization algebras on~$M$.

In the current paper, we provide a complementary perspective, in which we focus on computational tools for the study of the $\infty$-category of prefactorization algebras.
Namely, we describe a notion of ``prefactorization algebra up to homotopy'', a notion of ``$\infty$-morphism of homotopy prefactorization algebras'', and a homotopy transfer theorem for such algebras.
Our main technique is the Koszul theory for operads, of which we provide an overview in~\ref{sec opsforfas} for those readers unfamiliar with this toolkit.

We can state the main results of this paper as follows:

\begin{thmABC}[See Definition~\ref{def disj}, Theorem~\ref{thm: disjkoszul}]\label{thmA}
  Given a manifold $M$, there is a \emph{Koszul} quadratic-linear colored operad $\Disj_M$ whose category of algebras is the category of prefactorization algebras on $M$.
\end{thmABC}

Theorem~\ref{thmA} implies that there is a relatively simple cofibrant resolution $\hoDisj_M$ of $\Disj_M$, and we make this cofibrant resolution explicit.
Furthermore, we make explicit the category of algebras over this operad.

\begin{thmABC}[See Proposition~\ref{prop: hodisjexplicit}]\label{thmB}
  Given a manifold $M$, algebras over the Koszul resolution $\hoDisj_M = \Omega(\Disj_M^\antishriek)$ are called homotopy prefactorization algebras on $M$.
  Such an algebra is given by a collection $\cA = \{ \cA(U) \}_{U \subseteq M}$ indexed by open subsets of $M$, equipped with maps:
  \begin{equation*}
    \mu_{\cU}: \cA(U_{11})\otimes \cdots \otimes \cA(U_{k1})\to \cA(U_{1s_1}\sqcup \cdots \sqcup U_{ks_k})
  \end{equation*}
  for every collection $\cU = \bigl( U_{11} \subset \dots \subset U_{1s_1}, \dots, U_{k1} \subset \dots \subset U_{k s_k}\bigr)$ (with pairwise disjoint $U_{j s_j}$), satisfying several compatibility conditions (most notably Equation~\eqref{eq diff mu}).
\end{thmABC}

The exact relations satisfied by the $\mu_\cU$ are not tremendously important; the important thing is that Proposition~\ref{prop: hodisjexplicit} makes these relations completely explicit.
Finally, we discuss the homotopy transfer theorem for $\hoDisj_M$ algebras.
Readers may be familiar with the idea that, given a dg associative algebra $A$, its cohomology $H^\bullet(A)$ is also a dg associative algebra; however, there is not necessarily a map of dg associative algebras $H^\bullet(A)\to A$.
Instead, one must consider the data $H^\bullet(A)$ together with its \emph{Massey products}, which are at-least-trilinear operations on $H^\bullet(A)$.
These Massey products endow $H^\bullet(A)$ with the structure of an $A_\infty$-algebra, and this $A_\infty$-algebra is equivalent to $A$ in a suitable sense.
An analogous story is true for algebras over any Koszul operad, and in fact, as a consequence of Theorem~\ref{thmA}, we find:

\begin{thmABC}[See Proposition~\ref{prop HTT}]\label{thmC}
  Given a prefactorization algebra $\cA$ on $M$ and deformation retractions
  \[
    \ourDR{\cB(U)}{\cA(U)}{i_U}{p_U}{h_U}
  \]
  for all open subsets $U\subset M$, then the collection $\{\cB(U)\}_{U\in \PCosh(M)}$ has a $\hoDisj$-algebra structure and the maps $i_U$ can be extended into an $\infty$-morphism of $\hoDisj$-algebras.
\end{thmABC}

Though we have restricted our study in this paper to prefactorization algebras, the tools developed here can also be used to make judgements concerning the category of factorization algebras.
Given factorization algebras $A,B$, an $\infty$-quasi-isomorphism $A\rightsquigarrow B$ of $\hoDisj$-algebras produces (see Theorem 11.4.9 of \autocite{LodayVallette2012}) a zig-zag of quasi-isomorphisms of prefactorization algebras
\[
  A \leftarrow \bullet \to \bullet \leftarrow \bullet \cdots \bullet \to B.
\]
A simple argument shows that if one of $F$ or $G$ is a factorization algebra and $F\to G$ is a quasi-isomorphism of prefactorization algebras, then the other one is also a factorization algebra.
Hence, given factorization algebras $A$ and $B$ and an $\infty$-quasi-isomorphism $A\rightsquigarrow B$ (of homotopy prefactorization algebras), we obtain an equivalence between $A$ and $B$ in the homotopy category of factorization algebras.
In other words, because the category of factorization algebras is a full subcategory of that of prefactorization algebras, to the extent that the operadic results presented here allow one to make calculations in the homotopy category of prefactorization algebras, they also allow for calculations in the homotopy category of \emph{factorization algebras}.

In a similar vein, the results presented here can be used to study deformations of factorization algebras.
Namely, given a (homotopy) prefactorization algebra $A$, the techniques of Section 12.2 of \autocite{LodayVallette2012} provide a dg Lie algebra $\fg^A$ whose Maurer-Cartan locus can be identified with the space of deformations of $A$.
Recall that, given a dg-Artinian ring $(R, \mathfrak m)$, an $R$-deformation of $A$ is an $R$-linear (homotopy) prefactorization algebra structure on $A\otimes R$ which agrees with the original structure modulo the maximal ideal $\mathfrak m$.
Note that if $A$ is a factorization algebra, then any deformation of $A$ will also be a factorization algebra, for the following reason.
The codescent condition for $A$ can be rephrased as the acyclicity of a complex $C(A)$ determined functorially from the prefactorization-algebraic structure maps of $A$.
Given an $R$-deformation of $A$, we therefore need to study the complex $C(A\otimes R)$, which has a filtration by powers of the maximal ideal $\mathfrak m$.
On the $E^0$-page of the associated spectral sequence, we obtain the cohomology of $C(A)$, which is trivial by assumption.
Hence, any deformation $A\otimes R$ of $A$ will be a factorization algebra if $A$ is.
In short, the tools developed here can be used to study the formal deformations of factorization algebras (and not just prefactorization algebras).
\subsection{Outline}

The paper is organized as follows. In Section~\ref{sec op disj}, we present the operad $\Disj_M$, and we give a simple generators-and-relations presentation thereof.
Because we imagine a mixed audience of operadic and factorization-algebraic readers, we start the section with gentle introductions for each type of reader to the other field (cf. Sections~\ref{sec fact alg for op} and~\ref{sec opsforfas}).
Next, in Section~\ref{sec disj koszul}, we prove Theorem~\ref{thmA}.
The reader interested in ``the punchline'' may skip straight to Section~\ref{sec descr htpy prefac}, in which we make explicit the definition of a $\hoDisj_M$-algebra and $\infty$-morphism thereof.
We also state Theorem~\ref{thmC} in Section~\ref{sec descr htpy prefac}.
Finally, in Section~\ref{sec examples}, we describe $\hoDisj$-algebras on some finite spaces, including the empty set, the one-point space, and the Sierpiński space.
We also extend a theorem of \textcite{CG1} concerning the factorization enveloping algebra $\tilde \cF_\fg$, which is a prefactorization algebra on $\RR$.

\subsection{Future Directions}

The Koszul property (and its proof) of the operad $\Disj_M$ opens up several avenues for future research.

The first is the study of the deformation complex of prefactorization algebras, a kind of ``homology theory'' naturally defined for algebras over (Koszul) operads.
Prefactorization algebras have properties reminiscent of commutative algebras (see Section~\ref{sec examples}).
Locally constant unital prefactorization algebras on $M = \RR^n$ are equivalent to $E_n$-algebras, i.e., homotopy associative and commutative algebras.
The deformation complex of an $E_n$-algebra satisfies the Hochschild--Kostant--Rosenberg (HKR) theorem: it splits as the (symmetric) algebras of (shifted) polyvector fields on $\RR^n$.
For $n=1$, this is the classical HKR theorem~\cite{HochschildKostantRosenberg1962}; for higher $n$, this is due to \autocite{CalaqueWillwacher2015}.
It is natural to wonder whether the deformation complex of a prefactorization algebra on $M$ splits as the algebras of polyvector fields on $M$.
Theorem~\ref{thmB} provides a tool to study this question, as it can be used to give an explicit description of the deformation complex of a prefactorization algebra on $M$.

The second avenue is the relationship with known categories of prefactorization algebras.
As mentioned above, (locally constant) prefactorization algebras on the real line $M = \RR$ are equivalent to homotopy associative algebras (i.e., $E_1$-algebras or $A_\infty$-algebras).
Given a strictly associative algebra, it is easy to produce a prefactorization algebra on $\RR$, i.e., a $\Disj_\RR$-algebra.
However, given a specific $A_\infty$-algebra, it is not clear how to construct a prefactorization algebra on $\RR$ explicitly.
Using our methods, we hope to be able to describe an explicit equivalence that produces a $\hoDisj_\RR$-algebra from an $A_\infty$-algebra in a way that covers the functors which turns an associative algebra into a $\Disj_\RR$-algebra.

Finally, one is generally interested in prefactorization algebras that are equipped with more structure.
For example, in the deformation quantization approach to quantum field theory, $\PP_0$-prefactorization algebras---whose values are shifted Poisson algebras and whose structure maps are morphisms of Poisson algebras---are of particular interest (see the introduction of~\cite{CG1}).
However, the natural extension of the presentation of $\Disj_M$ to $\PP_0$-prefactorization algebras is not quadratic-linear.
Indeed, the property of being a morphism of algebras is cubic.
This failure can be fixed e.g., by introducing new generators to the presentation, but the computational difficulty quickly becomes prohibitive.
Nevertheless, a resolution of $\PP_0$-prefactorization algebras up to homotopy would be a valuable tool.

\begin{uremark}
  In this article, we solely focus on the operad which encodes (homotopy) prefactorization algebras.
  It is, however, also important to choose a good target (symmetric monoidal) category, e.g., graded vector spaces or chain complexes.
  Prefactorization algebras are most valuable when this target category involves topological vector spaces.
  A suitable target category must have a computable derived tensor product to be useful for applications.
  There are several constructions of suitable target categories, including differentiable vector spaces~\cite[App.~B]{CG1} and liquid vector spaces~\cite{Scholze2019,ClausenScholze2022}.
\end{uremark}

\subsection{Conventions}\label{sec conventions}

\begin{enumerate}[nosep]
  \item We fix a manifold $M$ throughout, except possibly in Section~\ref{sec examples}. We may or may not make $M$ explicit in the notation; for example, $\PCosh_M$ and $\PCosh$ will both refer to the poset (or category, or colored operad) of open subsets of $M$.
  \item Throughout, we fix the ground field to be $\RR$, although the results presented here apply equally  well for any field of characteristic zero.
  \item Fixing a finite set $X$ whose elements are $x_1,\ldots, x_N$, we let $\RR\{x_1,\ldots, x_N\}$ denote the free vector space on $X$.
  \item The symbol $\SS_k$ denote the symmetric group on $k$ letters.
  \item A rooted tree $T$ is a set $V$ of vertices, a pointed set $(H,r)$ of half-edges, a map $\mathrm{inc}: H\to V$ which takes each half-edge to the vertex on which it is incident, and an involution $\sigma: H\to H$ such that $\sigma(r)=r$. The orbits of $\sigma$ are called ``edges''; the orbits of cardinality two are called ``internal edges'', $r$ is called ``the root'', and all other fixed points of $\sigma$ are called the ``leaves'' of $T$. When we draw a tree, we draw it with the leaves at the top and the root at the bottom. In the operadic analogy, the trees represent compositions from top to bottom.
  \item By a(n $\PCosh$-colored) $\SS$-module, we mean a collection $\{E(k)\}_{k \geq 0}$ of vector spaces such that each $E(k)$ is a right $\SS_k$-module. Furthermore, each $E(k)$ is required to carry a decomposition into a direct sum of vector spaces
        \begin{equation*}
          E(k)=\bigoplus_{U_1,\ldots, U_k,V\in \PCosh}E\binom{V}{U_1,\ldots, U_k}
        \end{equation*}
        such that the action of $\SS_k$ on $E(k)$ is compatible with the action of $\SS_k$ on the set $\{1,\ldots, k\}$, namely the action of $\sigma\in \SS_k$ is completely determined by isomorphisms
        \begin{equation*}
          E\binom{V}{U_1,\ldots, U_k} \to E\binom{V}{U_{\sigma^{-1}(1)},\ldots, U_{\sigma^{-1}(k)}}.
        \end{equation*}
  \item If $E$ is a graded $\PCosh$-colored $\SS$-module, $E[1]$ is the graded $\PCosh$-colored $\SS$-module whose degree $p$ components are the degree $p+1$ components of $E$. $E\{1\}$ is the following $\SS$-module:
        \begin{equation*}
          E\{1\}(k) = E(k)\otimes \sgn_k[k-1].
        \end{equation*}
        If $E$ is an operad or cooperad, $E[1]$ is not necessarily also one, but $E\{1\}$ is.
  \item \label{it not operads} Most of our notation matches that of \autocite{LodayVallette2012} except that we use cohomological grading instead of homological grading and 2) we use the symbol $E\{1\}$ to denote what in the cited reference is referred to as $\mathscr S^{-1}E$. In particular, the notation $\mu\circ_i \nu$ denotes the composition of $\nu$ in the $i$-th position in $\mu$ (with identities elsewhere), and for a quadratic-linear operad $\cP =\cT(E,R)$ (with $R\subset \cT^{(2)}(E)\oplus E$), we let $qR$ denote the projection of $R$ on $\cT{(2)}(E)$ and denote by $q\cP$ the quadratic operad $\cT(E,R)$. In certain cases, the Koszul dual $q\cP^\antishriek$ has a differential induced from the structure of $\cP$, and we let $q\cP^\antishriek$ denote both the quadratic co-operad and the quadratic dg-cooperad.
  \item Given two $\SS$-modules $E,F$, we let $E\circ F$ denote the $\SS$-module such that
        \begin{equation*}
          E\circ F (k) = \bigoplus_{k_1+\cdots+k_p=n} \left( E(p)\otimes\left( \mathrm{Ind}_{\SS_{k_1}\times \cdots \times \SS_{k_p}}^{\SS_n} (F(k_1)\otimes \cdots \otimes F(k_p))\right)\right)^{\SS_p},
        \end{equation*}
        where the superscript notation denotes taking the invariants.
        One could make a definition of $E\tilde \circ F$ analogously, instead using the \emph{co}induced representation and the $\SS_p$ \emph{co}invariants; this definition is more natural for definitions involving cooperads, but since we are in characteristic 0, there is a natural isomorphism $E\tilde \circ F \cong E\circ F$, so we make no distinction between the two.
  \item Given an $\PCosh$-colored $\SS$-module $E$, we let $\cT(E)$ denote the free ($\PCosh$-colored) operad generated by $E$. Given a non-negative integer $k$, an $\RR$-linear basis for the space of $k$-ary operations in $\cT(E)$ is given by the set of isomorphism classes of shuffle trees on $E$ with $k$ leaves. A shuffle tree on $E$ is a rooted planar tree $\gamma$ together with a labeling of the leaves of $\gamma$ by the integers $\{1,\ldots, k\}$, with a condition on this labeling which we now explain.
        By recursion,  this labeling induces a labeling of all the internal edges of $\gamma$ by integers, as follows: the output of vertex $v$ is labeled by the minimum of all labels on inputs to $v$.
        A shuffle tree is a planar tree so labeled such that the labels on the inputs to a vertex are in increasing order from left-to-right.
        Finally, to take care of the colors and the module $E$, all the internal edges of $\gamma$ have colors from $\PCosh$, and vertices of $\gamma$ are labeled by elements of appropriate arity and color in $E$.
  \item Since we have fixed the spacetime manifold $M$ and we almost exclusively consider colored operads whose colors are $\PCosh$, we reserve the right to use the term ``operad'' when ``$\PCosh$-colored operad'' is more precise.
\end{enumerate}

\subsection{Acknowledgments}

The authors would like to thank Owen Gwilliam, Guillaume Laplante-Anfossi, and Stephan Stolz for helpful discussions and questions.
E.R.\ would also like to thank Damien Calaque, Víctor Carmona, and Thomas Willwacher for helpful discussions.
N.I.\ thanks the University of Notre Dame for its hospitality.

N.I. was supported by the project HighAGT (ANR-20-CE40-0016), the project SHoCoS (ANR-22-CE40-0008), and the IdEx Université Paris Cité (ANR-18-IDEX-0001), funded by the \emph{Agence Nationale de la Recherche} (France).

\section{The operad encoding prefactorization algebras}\label{sec: background}

In this section, we describe the operad $\Disj$ which encodes prefactorization algebras (see Section~\ref{sec op disj}).
Before we do so, we will first recall the definition of prefactorization algebras, the definition of operads, and how they are related.
We fix once and for all a topological space $M$ which will not generally appear in the notation.

\subsection{Operads for factorization algebraists}\label{sec opsforfas}

There is no substitute for the excellent book of \textcite{LodayVallette2012} on the subject of algebraic operads, and we use the results of that book heavily in this paper.
Nevertheless, we will try to summarize the main points of the formalism for those who have not seen them before.

\subsubsection*{Operads, cooperads, bar and cobar constructions}

A (symmetric) operad $\cO$ is a collection of vector spaces (or chain complexes, or objects in a more general symmetric monoidal category) $\{\cO(k)\}_{k=0}^\infty$ such that $\cO(k)$ is a module for $\SS_{k}$ (the symmetric group on $k$ elements), together with a collection of maps
\begin{equation*}
  \circ: \cO(k)\otimes \cO(n_1)\otimes \cdots \otimes \cO(n_k)\to \cO(n_1+\cdots+n_k)
\end{equation*}
which are associative and respect the $\SS_n$-module structures in a natural way.
For example, one may form operads $\Com$, $\As$, and $\Lie$ which encode commutative, associative, and Lie algebras, respectively.
More explicitly, $\Com(k)$ is the trivial $\SS_k$-module, and the composition maps are the natural isomorphisms.
A number of similar definitions immediately present themselves.
For example, by reversing the direction of the arrows, one obtains the notion of a cooperad.
One may also impose a notion of unitality or counitality by privileging an element $\id$ of $\cO(1)$ that acts as an identity (co)-operation, i.e., in the operad case, $\circ(\mu,\id,\ldots, \id)=\mu$.
Finally, one may allow the output and each of the $k$ inputs of an operation $\mu\in \cO(k)$ to be labeled by elements of some fixed set of ``colors'' $S$,
and to require the composite of operations to exist only when the output labels of the elements of $\cO(n_1),\ldots \cO(n_k)$ match the input labels of the element of $\cO(k)$.
In this way, one obtains the notion of a ``colored operad.''
Henceforth, we will assume that the underlying symmetric monoidal category is $\Ch$, and when we say ``operad'' (resp. ``cooperad''), we will generally mean ``unital (resp. counital) colored operad (resp. cooperad) in the symmetric monoidal category of chain complexes''.

Given a cooperad $\cC$ and an operad $\cP$, the collection
\begin{equation*}
  \bigoplus_{k} \Hom_{\SS_k}(\cC(k),\cP(k))
\end{equation*}
has a natural structure of a dg Lie algebra (the symbol $\Hom$ denotes here the internal hom in chain complexes).
We thus obtain a natural bifunctor
\begin{equation*}
  \Tw(\_, \_): \CoOp \times \Op \to \mathrm{Set}
\end{equation*}
which takes the Maurer-Cartan elements of this dg Lie algebra (Maurer-Cartan elements in this dg Lie algebra have the special name ``twisting morphism'', hence the notation).
It turns out that, once some small restrictions are placed on the categories of operads and cooperads under consideration, there is a left-right adjoint pair of functors
\begin{equation*}
  \Omega : \CoOp \leftrightarrows \Op: B
\end{equation*}
which represent the bifunctor $\Tw$, i.e., there are natural isomorphisms
\begin{equation*}
  \Hom_{\Op}(\Omega \cC, \cP)\cong \Tw(\cC,\cP)\cong \Hom_{\CoOp}(\cC,B\cP).
\end{equation*}
(To be precise, the true domains and codomains of these functors are the categories of \emph{conilpotent} cooperads and \emph{augmented} operads.)
The counit of the adjunction, $\Omega B \cP \to \cP$, induces a quasi-isomorphism of operads.
The goal of the Koszul theory of operads is to provide a smaller resolution of $\cP$ via an operad of the form $\Omega \cC$, for some sub-cooperad $\cC\hookrightarrow B\cP$.

\subsubsection*{Koszul theory for associative algebras}

Before we describe the Koszul theory of general operads, it is worthwhile to consider the special case of operads for which $\cO(k)=0$ unless $k=1$.
This is the case of (dg) associative, unital algebras.

As we mentioned parenthetically above, the functor $B$ is defined only on the category of augmented operads; in the case of unital algebras, the operadic augmentation translates into an augmentation for associative algebras, i.e., we must consider algebras equipped with an algebra map $\epsilon: A\to \RR$.
So, given an augmented associative (possibly dg) algebra $(A,\mu, \epsilon)$, we may form the bar-cobar resolution $\Omega B A$, as for operads, and this algebra is generated by elements of the form
\begin{equation*}
  [a_1|\cdots|a_n]\in A^{\otimes n},
\end{equation*}
where each $a_i\in \ker(\epsilon)$ and $n$ is any non-negative integer.
Such an element has cohomological degree $1-n + \sum_{i}|a_i|$, and the bar-cobar differential is defined by the equation
\begin{align*}
  d[a_1|\cdots|a_n] & = \sum_{i=1}^{n-1}(-1)^{i+|a_1|+\cdots+ |a_{i}|}[a_1|\cdots |\mu(a_i ,a_{i+1})|\cdots |a_n] \\
                    & +\sum_{i=1}^{n-1}(-1)^{|a_1|+\cdots+|a_i|-i+1}[a_1|\cdots|a_i]\cdot [a_{i+1}|\cdots|a_n],
\end{align*}
where the symbol $\cdot$ denotes the multiplication in the semi-free algebra $\Omega B A$.
So, whereas in the algebra $A$ we had elements $a,b$ whose product is $\mu(a,b)$, in the algebra $\Omega B A$ we have the elements $[a],[b]$ whose product is $[a]\cdot[b]$ and the equation $d[a|b]= \pm([\mu(a,b)]-[a]\cdot [b])$.
In particular, if there is a relation of the form $\mu(a,b)= \mu(a',b')$ in $A$, this relation no longer holds on the nose in terms of the product $\cdot$ on $\Omega B A$; instead, one has
\begin{equation*}
  [a]\cdot [b]- [a']\cdot [b'] = \pm d\left( [a'|b']-[a|b]\right).
\end{equation*}
Furthermore, the algebra $\Omega B A$ has homotopies between the homotopies $[a|b]$ and homotopies between those homotopies and so on.

This resolution has the benefit of being well-defined for any algebra $A$.
Its drawback is that it is very large: even if $A$ is finitely generated, $\Omega B A$ is not.
Even more, in the bar-cobar resolution, even the space of generators of weight one is not necessarily finite-dimensional when $A$ is finitely generated (where the \emph{weight} of $[a_1|\cdots |a_n]$ is $n$).
To this end, it is desirable to look for smaller resolutions in case $A$ is known to be described by a simple set of generators and relations.
This is the goal of Koszul theory (see, e.g., \autocite{Priddy} for the quadratic-linear case), which applies to algebras of the form $\cT(V,R)$, where $V$ is a vector space, $R\subseteq V^{\otimes 2}\oplus V\oplus \RR$, and $\cT(V,R)$ is the algebra generated by $V$ subject only to the relations in $R$.
The purpose of Koszul theory is to find a sub-coalgebra $A^\antishriek \to BA$ such that the composite $\Omega A^\antishriek \to \Omega B A \to A$ is still a quasi-isomorphism.
In fact, for any algebra of the form $\cT(V,R)$, there is a natural candidate for $A^\antishriek$; one says that an algebra is \emph{Koszul} if the map $\Omega A^\antishriek\to A$ is a quasi-isomorphism.

The benefit of this construction is that if $A$ is described by quadratic relations, then $A^\antishriek$ is described by quadratic corelations, whereas $BA$ is freely cogenerated by $A[1]$.
This can dramatically reduce the number of generators of the semi-free resolution.
For example, if $A=\cT(x)$ is freely generated by an element $x$, then the bar-cobar resolution is freely generated by elements of the form
\begin{equation*}
  [x^{i_1}|\cdots |x^{i_n}];
\end{equation*}
by contrast, the resolution of $A$ given by Koszul theory is just $A$ again (since $A$ was free to begin with, there was no need to provide it with a new resolution).
In this case, $A^\antishriek = \RR\{1,\epsilon\}$ where $|\epsilon|=-1$ and $\Delta(\epsilon)= 1\otimes \epsilon + \epsilon \otimes 1$.
In general, $A^\antishriek$ need not be finite-dimensional even if $A$ is finitely-generated, so that the Koszul resolution of a finitely-generated algebra is not necessarily finitely-generated.
But it will always be the case that the weight grading on $B A$ descends to $A^\antishriek$ and that the weight-one component of $A^\antishriek$ is finite-dimensional if $V$ is.
(In the example of $A=T(x)$, something stronger is true: the weight one component of $\Omega B A$ is finite-dimensional, and $\Omega A^\antishriek$ is finitely-generated.)

\emph{Mutatis mutandis}, the preceding discussion applies equally well to more general operads.
Moreover, many of the standard operads--including the associative, commutative, and Lie operads--are candidates for the application Koszul theory, since e.g., the associativity relation $\mu(\mu(a,b),c)=\mu(a,\mu(b,c))$ is quadratic in $\mu$.
Indeed, Koszul theory for these three operads produces the $A_\infty$, $C_\infty$, and $L_\infty$ operads, respectively.
In Section~\ref{sec: background}, we show that the operad $\Disj$ encoding prefactorization algebras on a fixed manifold $M$ has a quadratic-linear generators-and-relations presentation (cf. Proposition~\ref{prop: gensrels}); the main idea is that the operad is generated by the unary operations $m_U^V$ for any inclusion $U\subset V$ and the binary operations $m_{U,V}^{U\sqcup V}$ for any disjoint pair of sets.

\subsection{Factorization algebras for operadists}\label{sec fact alg for op}

Let us now briefly introduce the notion of (pre)factorization algebras.
A more detailed discussion can be found in \autocite{Ginot2015,CG1}.

A bird's eye view of (pre)factorization algebras is that they are algebraic structure similar to (pre)cosheaves which are ``multiplicative'' in nature.
The core difference between (pre)factorization algebras and (pre)cosheaves is that, while the value of a (pre)cosheaf on a disjoint union of open sets is the direct sum of the values on the individual open sets, the value of a (pre)factorization algebra on a disjoint union of open sets is equipped with a map from the tensor product of the values on the individual open sets.
While this difference may seem innocuous, it has a profound effect on the structure of (pre)factorization algebras.

Cosheaves are used to define a cohomology theory of arbitrary topological spaces with coefficients in abelian groups (Čech cohomology).
On the other hand, factorization algebras are used to define a homology theory of manifolds with coefficients in algebras over the little disks operads, i.e., $E_n$-algebras (through factorization homology).
While Čech cohomology is homotopy invariant, factorization homology is not: homotopy equivalent manifolds may have differing factorization homology groups if they are not diffeomorphic.
Factorization algebras are thus much more sensitive to the geometry of the underlying space than cosheaves are.
This sensitivity is, of course, dependent on the choice of coefficients: if we take the space of coefficients to be a commutative algebra (or $E_\infty$-algebra), then factorization homology becomes homotopy invariant as it reduces to higher Hochschild homology~\cite{Pirashvili2000}.

This power comes at a cost, though: the structure of factorization algebras is much more complicated than that of cosheaves.
For example, the category of cosheaves on a topological space is additive, while the category of factorization algebras on a manifold is not.
The description of the descent property (on which we will not talk about in this article) satisfied by cosheaves is also much simpler than that of factorization algebras.
Even if we restrict ourselves to prefactorization algebras (i.e., we ignore the descent condition), then the structure of the category of prefactorization algebras is still much more complicated than that of cosheaves.
In operadic terms, our goal in this article is to explain that the operad governing prefactorization algebras is the composition product of the operads governing cosheaves and the operad governing a certain notion of colored commutative algebras (cf. Lemma~\ref{lem: compprod}).
We will prove separately that both components of the product are Koszul operads (cf. Lemmas~\ref{lem: openkoszul} and~\ref{lem: tenskoszul}), and prove that the distributive law between the two satisfies the diamond lemma, from which the Koszul property of the composition product follows (cf. Theorem \ref{thm: disjkoszul}).
We then deduce an explicit description of the Koszul dual (and thus of a resolution) of the operad governing prefactorization algebras (cf. Lemma~\ref{lem: infdecomp} and Proposition~\ref{prop: hodisjexplicit}).

\subsection{The operad \texorpdfstring{$\Disj$}{Disj}: definitions and basic properties}\label{sec op disj}

\begin{definition}\label{def disj}
  The colored $\RR$-linear operad $\Disj$ has as its colors the open subsets $U\subseteq M$.
  Given open subsets $U_1, \dots, U_k, V \subseteq M$, we set:
  \begin{equation*}
    \Disj{\binom{V}{U_1,\ldots, U_k}}
    \coloneqq
    \begin{cases}
      \RR\{m^V_{U_1, \dots, U_k}\}, & \text{if } U_i \subseteq V \text{ and the } U_i \text{ are pairwise disjoint}; \\
      \{0\},                        & \text{else}.
    \end{cases}
  \end{equation*}
  The composition map
  \begin{equation*}
    \textstyle
    \Disj\binom{W}{V_1,\ldots, V_k} \otimes \Disj\binom{V_1}{U_{11},\ldots, U_{1n_1}}\otimes \dots \otimes \Disj\binom{V_k}{U_{k1},\ldots, U_{kn_k}}
    \to
    \Disj\binom{W}{U_{11},\ldots, U_{kn_k}}
  \end{equation*}
  is zero when any of the factors in the domain or codomain is zero, and otherwise is the natural isomorphism $\RR^{\otimes (k+1)}\to \RR$.
  A permutation $\sigma\in \SS_k$ sends the generator $m_{U_1,\dots,U_k}^V$ to $m_{U_{\sigma^{-1}(1)},\dots,U_{\sigma^{-1}(k)}}^V$.
\end{definition}

\begin{definition}\label{def pfa}
  A \textbf{prefactorization algebra} is an algebra over $\Disj$.
\end{definition}

\begin{definition}\label{def pcosh}
  We let $\PCosh$ denote the poset of open subsets of $M$.
  This poset defines a category, i.e., a colored operad with only unary operations.
\end{definition}

Note that $\PCosh$ is a sub-operad of $\Disj$.
Occasionally, when we want to make explicit the underlying space $M$, we will also write $\Disj_M$ or $\PCosh_M$.
A prefactorization algebra consists of
\begin{enumerate}
  \item A cochain complex $\mathcal F(U)$ for every open subset $U \subseteq M$.
  \item A map
        \begin{equation*}
          m_{U_1,\ldots, U_k}^V: \cF(U_1)\otimes \cdots \cF(U_k) \to \cF(V)
        \end{equation*}
        for any collection $\{U_i\}$ of pairwise disjoint open subsets of $V$.
\end{enumerate}
These data are required to satisfy the following relations:
\begin{description}
  \item[Symmetry] Given a permutation $\sigma\in \SS_k$, the following diagram
    \begin{equation*}
      \begin{tikzcd}[row sep = large]
        \cF(U_1)\otimes\cdots \otimes \cF(U_k)
        \ar[r,"\sigma"]
        \ar[rd,"m_{U_1,\ldots, U_k}^V"']
        & \cF(U_{\sigma(1)})\otimes \cdots \otimes \cF(U_{\sigma(k)})
        \ar[d,"m_{U_{\sigma(1)},\ldots, U_{\sigma(k)}}^V"]\\
        & \cF(V)
      \end{tikzcd}
    \end{equation*}
    commutes, where the horizontal arrow is induced from the symmetric monoidal structure on the category of cochain complexes (including the usual Koszul signs).
  \item[Associativity] Given any pairwise disjoint collection $\{V_j\}_{j=1}^r$ of subsets of $W$ and, for each $j$, a collection $\{U_{ji}\}_{i=1}^{k_j}$ of pairwise disjoint subsets of $V_j$, the following diagram commutes:
    \begin{equation*}
      \begin{tikzcd}[column sep = 8em, row sep = large]
        \cF(U_{11})\otimes\cdots \otimes \cF(U_{rk_r})
        \ar[r,"\bigotimes_{j=1}^r m_{U_{j1},\ldots, U_{jk_j}}^{V_j}"]
        \ar[rd,"m_{U_{11},\ldots, U_{rk_r}}^W"']
        & \cF(V_1)\otimes \cdots \otimes \cF(V_r)
        \ar[d,"m_{V_1,\ldots, V_r}^W"]\\
        & \cF(W)
      \end{tikzcd}
    \end{equation*}
\end{description}

\begin{remark}\label{rmk maps to com}
  There are maps of colored operads
  \begin{equation*}
    \Disj \to \Com,\quad \Com\to \Disj.
  \end{equation*}
  The first map covers the unique map of labels $\PCosh\to \{\ast\}$.
  The second map covers the map of sets of labels $\{\ast\}\to \PCosh$ which ``picks out'' the empty set amongst the open subsets of $M$.
  At the level of algebras, the first map of operads extracts from a prefactorization algebra $\cF$ the commutative algebra $\cF(\emptyset)$, while the second map of operads assigns to a commutative algebra $A$ the prefactorization algebra $\cF_A$ such that $\cF_A(U)=A$ for all open subsets $U\subseteq M$.
  We will see that, on account of these comparisons (in particular the latter one), the operad $\Disj$ behaves similarly to the commutative operad.
\end{remark}

Now, we give a generators-and-relations presentation of $\Disj$ which will enable us to compute its cofibrant resolution using Koszul duality theory.
\begin{definition}
  Given open sets $U \subset V$, we let $\ext \coloneqq m_U^V$ denote the (unary) generator of $\Disj\binom{V}{U}$.
  Given two disjoint open sets $U$ and $V$, we let $\bin \coloneqq m_{U,V}^{U \sqcup V}$ denote the (binary) generator of $\Disj\binom{U\sqcup V}{U,V}$.
\end{definition}

\begin{lemma}\label{lem disj generators}
  The operad $\Disj$ is generated by the operations of the form $\ext$ and $\bin$.
\end{lemma}
\begin{proof}
  This follows from the following equation, which holds for any pairwise disjoint open sets $U_1, \dots, U_k$ contained in an open set $V$:
  \begin{equation}\label{eq: generators}
    m_{U_1,\ldots, U_k}^V
    = m_{U_1\sqcup\cdots\sqcup U_k}^V\circ_1 m_{U_1\sqcup\cdots \sqcup U_{k-1},U_k}^{U_1\sqcup \cdots \sqcup U_k}\circ_1 m_{U_1\sqcup \cdots\sqcup U_{k-2},U_{k-1}}^{U_1\sqcup\cdots\sqcup U_{k-1}}\circ_1\cdots \circ_1 m_{U_1,U_2}^{U_1\sqcup U_2},
  \end{equation}
  where, as in \autocite{LodayVallette2012}, the notation $\circ_1$ means composition in the first slot/position.\qedhere
\end{proof}

\begin{definition}
  Let $E$ be the $\PCosh$-colored $\SS$-module defined as follows:
  \begin{align*}
    E(1) & = \bigoplus_{U\subsetneq V} \RR\{ \ext\},
         &
    E(2) & = \bigoplus_{U\cap V = \emptyset} \RR\{\bin\},
         &
    E(k) & = 0 \text{ for } k \not\in \{1,2\}.
  \end{align*}
\end{definition}

We may form the free $\PCosh$-colored operad $\cT(E)$ on $E$.
This free operad is weight-graded, by the number of generating operations (i.e., the number of vertices in the trees).
We let $\cT^{(k)}(E)$ denote the weight $k$ part of $\cT(E)$.

\begin{definition}
  Fix an ordered triple $(U,V,W)$ of open subsets of $M$.
  Following \autocite[Section~7.6.3]{LodayVallette2012}, we obtain the following $\RR$-linear generators for the subspace of $\cT^{(2)}(E)$ corresponding to trees whose vertices are labeled by the binary generators $\bin[\_][\_]$:
  \begin{align*}
    \taui{I}   & \coloneqq \bin[U\sqcup V][W] \circ_1\bin, \\
    \taui{II}  & \coloneqq \bin[W\sqcup U][V]\circ_1 \bin[U][W]\cdot (23), \\
    \taui{III} & \coloneqq \bin[U][V\sqcup W]\circ_2 \bin[V][W].
  \end{align*}
  Figure \ref{fig: weighttwotrees} depicts these operations.
\end{definition}

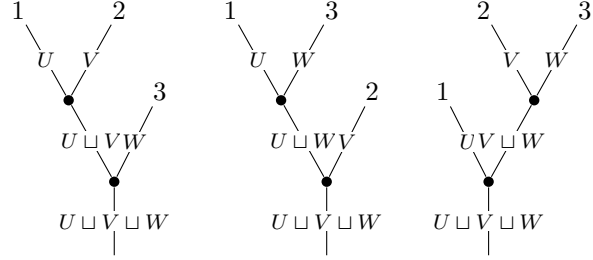
\begin{figure}[htbp]
  \centering
  \begin{forest}{operad}
    [[,lb={$U \sqcup V \sqcup W$}
            [,lb={$U \sqcup V$}
                [1, lb={$U$}]
                [2, lb={$V$}]]
            [3, lb={$W$}]
        ]]
  \end{forest}
  \quad
  \begin{forest}{operad}
    [[,lb={$U \sqcup V \sqcup W$}
            [,lb={$U \sqcup W$}
                [1, lb={$U$}]
                [3, lb={$W$}]]
            [2, lb={$V$}]
        ]]
  \end{forest}
  \quad
  \begin{forest}{operad}
    [[,lb={$U \sqcup V \sqcup W$}
            [1, lb={$U$}]
            [,lb={$V \sqcup W$}
                [2, lb={$V$}]
                [3, lb={$W$}]]
        ]]
  \end{forest}
  \caption{The generators $\taui{I}$,  $\taui{II}$, and $\taui{III}$ for the space of operations which have weight two and use only binary generators. All three are operations of color $\binom{U\sqcup V\sqcup W}{U,V,W}$.}\label{fig: weighttwotrees}
\end{figure}

\begin{remark}
  To give a sense of how the symmetric group acts on these generators, note the relation:
  \begin{equation*}
    \taui{I}\cdot (123) = \taui{II}[U,W,V]
  \end{equation*}
\end{remark}

\begin{definition}
  Let $R \coloneqq R_1\oplus R_2\oplus R_3$ denote the sub-$\SS$-module of $\cT^{(2)}(E)\oplus E$ defined by:
  \begin{align*}
    R_1 & \coloneqq \bigoplus_{U \subset V \subset W} \RR\left\{ \ext[V][W]\circ \ext - \ext[U][W]\right\}, \\
    R_2 & \coloneqq \bigoplus_{\substack{(U,V,W) \\ U\cap V = \emptyset\\ U\cap W = \emptyset\\ V\cap W=\emptyset}}\RR\left\{\taui{I}-\taui{II},\taui{II}-\taui{III}\right\}\\
    R_3 & \coloneqq \bigoplus_{\substack{U\subset V \\V\cap W=\emptyset}}\RR\{\bin[V][W]\circ_1\ext-\ext[U\sqcup W][V\sqcup W]\circ_1\bin[U][W]\}\\
        & \oplus \bigoplus_{\substack{W\subset V \\V\cap U=\emptyset}}\RR\{\bin\circ_2 \ext[W][V]- \ext[U\sqcup W][U\sqcup V]\circ_1 \bin[U][W]\}.
  \end{align*}
\end{definition}

We have written $R$ in this way to highlight a few major points.
Note that the space of relations is split depending on the type of generators used (exclusively unary, exclusively binary, or mixed).
The space $R_1$ consists of unary operations, $R_2$ of ternary operations, and $R_3$ of binary operations.
Moreover, let us note that we have the following inclusions:
\begin{align*}
  R_1 & \subseteq \cT^{(2)}(E) \oplus E, & R_2, R_3 & \subseteq \cT^{(2)}(E).
\end{align*}
In other words, the relations in $R_2$ and $R_3$ are homogeneous quadratic in the generators, whereas $R_1$ contains quadratic-linear relations, i.e., sums of generators and compositions of two generators.

\begin{proposition}\label{prop: disjpcosh}
  The operad $\cT(E(1), R_1)$ (generated by the $\ext$, subject to the relations in $R_1$) is isomorphic to $\PCosh$, the operad encoding precosheaves on $M$ (Definition~\ref{def pcosh}).
\end{proposition}

\begin{proof}
  A generic element of $\cT(E(1), R_1)$ is of the form
  \begin{equation*}
    \ext[U_k][U_{k-1}] \circ \dots \circ \ext[U_1][U_0],
  \end{equation*}
  with $U_0 \subset \dots \subset U_k$.
  Applying repeatedly the relations in $R_1$, we see that this element is equal to $\ext[U_k][U_0]$.
  The morphism of operads $\cT(E(1), R_1) \to \PCosh$ is thus bijective.
\end{proof}


\begin{definition}
  Let $\Tens$ be the operad $\cT(E(2), R_2)$, i.e., the colored binary operad generate by elements of the form $\bin$ subject to the relations in $R_2$.
  We will call algebras over $\Tens$ \textbf{tensor systems}.
\end{definition}

This colored operad resembles a colored version of the commutative operad.
The relation is made precise in Section~\ref{sec finite}.

The space $R_3$ encodes the relations between the binary generators and the unary ones.
Given any composable unary operation $\iota$ and binary operation $\mu$, the relations in $R_3$ tell us how to rewrite the composites $\mu\circ_1 \iota$ and $\mu\circ_2 \iota$ in terms of a composition where the binary operation is performed first.
We we will see (Lemma~\ref{lem: compprod}) that this gives a rewriting rule: namely, one can write
\begin{equation*}
  \cT(E,R) \cong \PCosh \circ \Tens.
\end{equation*}
We will see in the next proposition that $\cT(E,R) \cong \Disj$ (Proposition~\ref{prop: gensrels}).
Hence, we can conclude that $\Disj$ can be written as a composition product of the operads $\PCosh$ and $\Tens$.

Having established the notation, we may prove:
\begin{proposition}\label{prop: gensrels}
  There is an isomorphism $\Phi: \cT(E,R)\to \Disj$ of $\PCosh$-colored operads.
\end{proposition}

Before embarking on the proof of the proposition, let us describe the structure maps $\ext$ and $\bin$ in two common cases.
In the first case, given a commutative algebra $(C,\mu)$, we can construct the prefactorization algebra $\cF_C$ which assigns $C$ to any open subset $U\subseteq M$ ($M$ can be an arbitrary space).
In this case, we have $\ext = \id_C$ and $\bin=\mu$, so the maps $\ext$ are ``boring'' and the maps $\bin$ are ``interesting''.
In the second common case, let $(A,\mu,\eta)$ be an associative, unital algebra; one may form the factorization algebra $\cF_A$ on $\RR$ which assigns $\otimes_{\pi_0(U)}A$ to any open subset $U$ of $\RR$.
In this case, the maps $\ext$ are ``interesting'' and the maps $\bin$ are ``boring''.
More precisely, the map $\bin$ is given by the natural associator isomorphism
\begin{equation*}
  \left(\bigotimes_{\pi_0(U)}A\right)\otimes \left(\bigotimes_{\pi_0(V)}A\right)\to \bigotimes_{\pi_0(U\sqcup V)}A,
\end{equation*}
while the map $\ext$ is determined by the structures $\mu$ and $\eta$ on $A$, e.g.,
\begin{equation*}
  \ext[\emptyset][(0,1)]=\eta, \quad \ext[(-1,0)\sqcup(0,1)][(-1,1)]=\mu.
\end{equation*}

In the previous paragraph, the assignment $C\mapsto \cF_C$ gives a functor $\Com\text{-}\Alg\to \Disj\text{-}\Alg$; and it turns out that this assignment is induced via pullback from a map of operads $\Disj\to \Com$.
There is also a map of operads the other way, which at the level of algebras sends a prefactorization algebra $\cF$ to its value $\cF(\emptyset)$ on the empty set.
(It turns out that the two functors are adjoints, with the first functor the left adjoint.)
As we will see in Section~\ref{sec finite}, these relationships will extend to the resolutions: namely, let $\hoDisj$ denote the resolution $\Omega \Disj^\antishriek$ mentioned in the introduction (cf. Theorem~\ref{thmB}).
Similarly, let $C_\infty$ denote the resolution of $\Com$ obtained via Koszul theory.
Then, there is a pair of functors
\begin{equation}
  \label{eq: cinftyhodisj}
  C_\infty\text{-}\Alg\leftrightarrows \hoDisj_M\text{-}\Alg.
\end{equation}

By contrast, the functor $A\mapsto \cF_A$ which maps a unital, associative algebra to a prefactorization algebra on $\RR$ does not arise from a map of (colored) operads $\Disj\to u\As$ (the $u$ standing for ``unital'').
Indeed, the structure maps of the algebra $A$ appear in the unary generating operations of $\cF_A$, as opposed to the binary operations.
This suggests that the existence of the functor $A\mapsto \cF_A$ has more to do with the topology of open subsets of the real line than with general algebraic properties of the colored operad $\Disj$.
Consequently, though one may imagine the existence, by analogy with Equation \eqref{eq: cinftyhodisj}, of a functor
\begin{equation*}
  uA_\infty \text{-}\Alg\to \hoDisj_\RR\text{-}\Alg
\end{equation*}
(where $uA_\infty$ is the Koszul resolution of the operad governing unital associative algebras, see \autocite{HirshMilles}), such a functor does not immediately present itself using our methods.
There is, nevertheless, the composite functor
  \[
    uA_\infty \text{-}\Alg \to u\As\text{-}\Alg \to\Disj_\RR\text{-}\Alg \to \hoDisj_\RR\text{-}\Alg,
  \]
where the first functor is strictification, the second is the original functor, and the final functor is the embedding of prefactorization algebras in homotopy prefactorization algebras.
We believe that it is, however, unlikely that one would find the relations for $A_\infty$-algebras ``living inside'' the relations holding for $\hoDisj_\RR$-algebras.
To some extent, this is unsurprising, since our methods are meant to treat general prefactorization algebras, and not just locally constant ones.
We leave it as an open question to determine what sort of category of homotopy associative algebras forms the natural domain of a functor like the one above.

\begin{proof}
  We have already remarked upon a few parts of the proof.
  Define
  \begin{align*}
    \Phi(\bin) & \coloneqq m_U^V, & \Phi(\ext) & \coloneqq m_{U,V}^{U\sqcup V}.
  \end{align*}
  By the universal property of the free operad, these equations suffice to define a map $\cT(E)\to \Disj$, which we also denote using the letter~$\Phi$.
  It is immediate that $\Phi$ vanishes on $R$ and thus descends to a map $\cT(E,R)\to \Disj$.
  Equation \eqref{eq: generators} shows that $\Phi$ is arity-wise surjective.

  To show that $\Phi$ is arity-wise injective, consider an operation $\mu$ of color
  \begin{equation*}
    \binom{V}{U_1,\ldots, U_k}
  \end{equation*}
  in $\cT(E)$.
  The operation $\mu$ can be represented (non-uniquely) by a planar rooted tree with bivalent and trivalent vertices, together with a bijection from the set of leaves to~$\{1,\ldots, k\}$.
  Each edge, including those incident on the leaves and root, has a label from the poset~$\PCosh$.
  Furthermore, if the incoming edges of a trivalent vertex have colors $U$ and $V$, then the output has to have color $U\sqcup V$.
  We will show that, modulo $(R)$, $\mu$ is of the form appearing in the right-hand-side of Equation \eqref{eq: generators}.
  Let us observe first that any $\ext$ operation which is precomposed with a $\bin$ operation can be moved to the output of the $\bin$ operation using $R_3$ relations.
  Hence, modulo $(R)$, $\mu$ is equivalent to a rooted planar binary tree with a ``ladder'' attached to the root.
  Necessarily, the first input color on the ladder is $U_1\sqcup\cdots \sqcup U_k$, and the final input color on the ladder is $V$.
  Next, we use $R_1$ to replace this ladder of $\ext$ operations by $\ext[U_1\sqcup\cdots\sqcup U_k][V]$.
  Finally, we use the associativity relation in $R_2$ to turn the tree into a ``left comb.''
  When read from left to right, the labels on the leaves may not be in order.
  Here again, we may use $R_2$.
  Figure \ref{fig: rewriting} shows how this works in an example.

  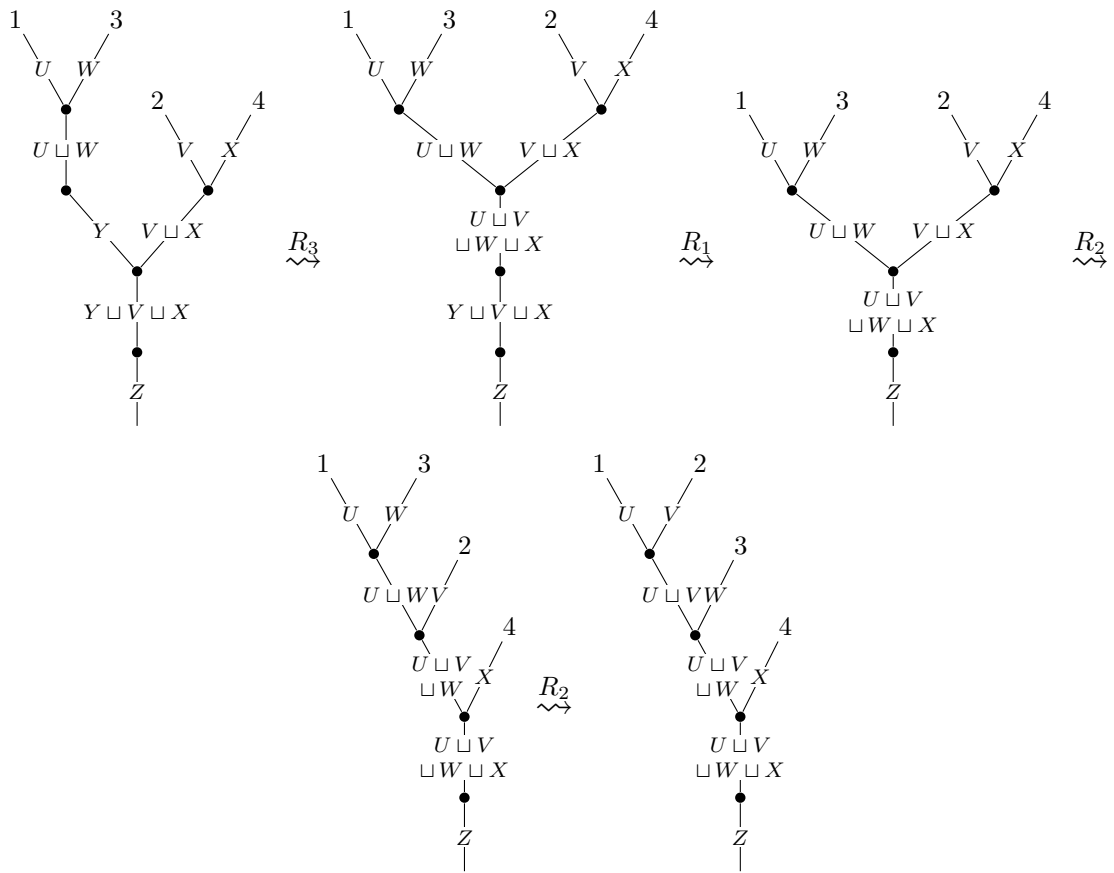
\begin{figure}[htbp]
    \centering
    \begin{forest}{operad}
      [[,lb={$Z$}
              [,baseline,lb={$Y \sqcup V \sqcup X$}
                  [,lb={$Y$}
                      [,lb={$U \sqcup W$}
                          [1, lb={$U$}]
                          [3, lb={$W$}]
                      ]
                  ]
                  [,lb={$V \sqcup X$}
                      [2, lb={$V$}]
                      [4, lb={$X$}]
                  ]
              ]
          ]]
    \end{forest}
    {\Large$\overset{R_3}{\leadsto}$}
    \begin{forest}{operad}
      [[,lb={$Z$}
              [,baseline,lb={$Y \sqcup V \sqcup X$}
                  [,lb={$U \sqcup V$\\$\sqcup \, W \sqcup X$}
                      [,lb={$U \sqcup W$}
                          [1, lb={$U$}]
                          [3, lb={$W$}]
                      ]
                      [,lb={$V \sqcup X$}
                          [2, lb={$V$}]
                          [4, lb={$X$}]
                      ]
                  ]
              ]
          ]]
    \end{forest}
    {\Large$\overset{R_1}{\leadsto}$}
    \begin{forest}{operad}
      [[,lb={$Z$}
              [,baseline,lb={$U \sqcup V$\\$\sqcup \, W \sqcup X$}
                  [,lb={$U \sqcup W$}
                      [1, lb={$U$}]
                      [3, lb={$W$}]
                  ]
                  [,lb={$V \sqcup X$}
                      [2, lb={$V$}]
                      [4, lb={$X$}]
                  ]
              ]
          ]]
    \end{forest}
    {\Large$\overset{R_2}{\leadsto}$}
    \begin{forest}{operad}
      [[,lb={$Z$}
              [,baseline,lb={$U \sqcup V$\\$\sqcup \, W \sqcup X$}
                  [,lb={$U \sqcup V$ \\ $\sqcup \, W$}
                      [,lb={$U \sqcup W$}
                          [1, lb={$U$}]
                          [3, lb={$W$}]
                      ]
                      [2, lb={$V$}]
                  ]
                  [4, lb={$X$}]
              ]
          ]]
    \end{forest}
    {\Large$\overset{R_2}{\leadsto}$}
    \begin{forest}{operad}
      [[,lb={$Z$}
              [,baseline,lb={$U \sqcup V$\\$\sqcup \, W \sqcup X$}
                  [,lb={$U \sqcup V$ \\ $\sqcup \, W$}
                      [,lb={$U \sqcup V$}
                          [1, lb={$U$}]
                          [2, lb={$V$}]
                      ]
                      [3, lb={$W$}]
                  ]
                  [4, lb={$X$}]
              ]
          ]]
    \end{forest}
    \caption{Applying the relations to reduce a tree to a standard form.}\label{fig: rewriting}
  \end{figure}

  In other words, in arity
  \begin{equation*}
    \binom{V}{U_1,\ldots, U_k},
  \end{equation*}
  the operad $\cT(E,R)$ is a vector space of dimension at most 1, and $\Phi$ is a surjective map from this space to a line.
  Hence, $\cT(E,R)$ must be a line in this arity, and $\Phi$ must be an isomorphism.
\end{proof}

Proposition~\ref{prop: gensrels} gives a quadratic-linear, generators-and-relations presentation of the operad $\Disj$.
In light of the proposition, we will cease to make a distinction between $\cT(E,R)$ and $\Disj$.
We would like to apply Koszul duality theory to the operad $\Disj$.
To do this, we need to check first that the colored operad $\Disj$ satisfies the properties \ref{ql1} and \ref{ql2} detailed in \autocite[Section~7.8]{LodayVallette2012}.
These conditions pertain to the minimality of the set of generators and the maximality of the relations; we must check them because our presentation is quadratic-linear.
Next, we need to check that the operad $\Disj$ is Koszul.
We undertake the first task here, and the second task below.

\begin{lemma}
  The operad $\Disj = \cT(E,R)$ satisfies the conditions for a quadratic-linear operad from~\autocite[Section~7.8]{LodayVallette2012}:
  \begin{enumerate}[label={(ql\arabic{*})}]
    \item\label{ql1} $R\cap E= \{0\}$;
    \item\label{ql2} $\{R\circ_{(1)} E + E\circ_{(1)} R\} \cap \cT^{(2)}(E)\subseteq R\cap \cT^{(2)}(E)$.
  \end{enumerate}
\end{lemma}

\begin{remark}
  The condition \ref{ql1} is a minimality condition for the generators: if the relations set some generators to zero, then those generators should be excluded to satisfy \ref{ql1}.
  The condition \ref{ql2} is a maximality condition for the set of relations.
  Indeed, $R\circ_{(1)}E+E\circ_{(1)}R$ consists of (cubic-quadratic) relations which hold in $\cT(E,R)$ as a consequence of the relations $R$.
  Since $R$ has quadratic-linear relations, some of these relations could be purely quadratic as cubic terms could cancel.
  However, the condition \ref{ql2} ensures that these quadratic relations are already in $R$.
\end{remark}

\begin{proof}
  The property \ref{ql1} is immediate.

  For \ref{ql2}, let us describe the plan of attack.
  We have described a basis for $E$ and a basis for $R$.
  This allows one to give a relatively straightforward description of the set
  \begin{equation*}
    S\coloneqq \{E\circ_{(1)}R+R\circ_{(1)}E\}.
  \end{equation*}
  In general, this is a subset of $\cT^{(3)}(E)\oplus \cT^{(2)}(E)$.
  The intersection of $S\cap\cT^{(2)}(E)$ can be found by identifying all cubic terms arising in $S$ and then studying how to cancel them with each other.
  After obtaining a cancellation in this way, we note any quadratic terms which remain and check whether or not these belong to $R$.

  We can group the possible cubic terms by the nature of their vertices:

  \begin{description}
    \item[Three extension maps $\ext$]
      The only possibility for such terms is by composing $R_1$ with a $\ext$ operation. In this way, we obtain---for any sequence $U\subseteq V\subseteq W\subseteq X$---the relations
      \begin{equation*}
        \ext[W][X]\circ \ext[V][W]\circ \ext[U][V]-\ext[V][X]\circ \ext[U][V],\quad \ext[W][X]\circ \ext[V][W]\circ \ext[U][V]-\ext[W][X]\circ \ext[U][W].
      \end{equation*}
      The only way to cancel the cubic terms above is to subtract the two terms; in this way, we obtain the relation
      \begin{equation*}
        \ext[V][X]\circ \ext[U][V]-\ext[W][X]\circ \ext[U][W]= \left(\ext[V][X]\circ \ext[U][V]-\ext[U][X]\right) - \left(\ext[W][X]\circ \ext[U][W]-\ext[U][X]\right);
      \end{equation*}
      this term manifestly lies in $R_1$.

    \item[Three binary operations $\bin$]
      These can only be obtained from composing $R_2$ with a binary generator.
      Since $R_2$ is purely quadratic, such terms all lie in $\cT^{(3)}(E)$, with no terms in $\cT^{(2)}(E)$.
      Hence,
      \begin{equation*}
        \{R_2\circ_{(1)}\RR\{\bin\}+\RR\{\bin\}\circ_{(1)} R_2\}\cap \cT^{(2)}(E)=0.
      \end{equation*}

    \item[Two binary operations $\bin$]
      These can be generated by a relation in $R_3$ composed with a binary generator, or a relation in $R_2$ composed with a unary generator. Since $R_2$ and $R_3$ are purely quadratic, this case is dealt with as in the case immediately preceding.

    \item[Two unary operations $\ext$]
      This is the only case that presents difficulties. Such terms are created by composing an $R_1$ relation with a binary generator or an $R_3$ relation with a unary generator.
      Let us give a full accounting of the composites that may appear in this way, fixing the two input sets $U,V$ and the output set $W$.
      We obtain the following spanning set for the subspace $S'$ of $S$ whose cubic terms involve exactly two unary operations:
      \begin{align*}
        \ext[W'][W]\circ\ext[U\sqcup V][W']\circ\bin             & - \ext[U\sqcup  V][W]\circ \bin
        \tag{$e_1$} \\
        \bin[U''][V]\circ_1 (\ext[U'][U'']\circ\ext[U][U'])      & - \bin[U''][V]\circ_1 \ext[U][U''],                                    & (W=U''\sqcup V)
        \tag{$e_2$} \\
        \bin[U][V'']\circ_2(\ext[V'][V'']\circ\ext[V][V'])       & -\bin[U][V'']\circ_2 \ext[V][V''],                                     & (W=U\sqcup V'')
        \tag{$e_3$} \\
        (\bin[U'][V']\circ_1 \ext[U][U'])\circ_2 \ext[V][V']     & - \ext[U\sqcup V'][U'\sqcup V'] \circ (\bin[U][V']\circ_2\ext[V][V']), & (W=U'\sqcup V')
        \tag{$e_4$} \\
        \bin[U''][V]\circ_1(\ext[U'][U'']\circ\ext[U][U'])       & -\ext[U'\sqcup V][U''\sqcup V]\circ (\bin[U'][V]\circ_1\ext[U][U']),   & (W=U''\sqcup V)
        \tag{$e_5$} \\
        \ext[U'\sqcup V][W]\circ (\bin[U'][V]\circ_1\ext[U][U']) & -\ext[U'\sqcup V][W]\circ\ext[U\sqcup V][U'\sqcup V]\circ\bin[U][V]
        \tag{$e_6$} \\
        (\bin[U'][V']\circ_1 \ext[U][U'])\circ_2 \ext[V][V']     & - \ext[U\sqcup V'][U'\sqcup V'] \circ (\bin[U'][V]\circ_1\ext[U][U']), & (W=U'\sqcup V')
        \tag{$e_7$} \\
        \bin[U][V'']\circ_2(\ext[V'][V'']\circ\ext[V][V'])       & -\ext[U\sqcup V'][U\sqcup V'']\circ (\bin[U][V']\circ_2\ext[V][V']),   & (W=U\sqcup V'')
        \tag{$e_8$} \\
        \ext[U\sqcup V'][W]\circ (\bin[U][V']\circ_2\ext[V][V']) & -\ext[U\sqcup V'][W]\circ\ext[U\sqcup V][U\sqcup V']\circ\bin[U][V]
        \tag{$e_9$}
      \end{align*}
      Here, we require the strict inclusions $U\subsetneq U'\subsetneq U''$, $V\subsetneq V'\subsetneq V''$, and $W'\subsetneq W$, and a set of generators for $S'$ is obtained by letting $U', V',\ldots$ vary over all such.
      Some of the generators only appear if $W$ is of a specific form, as noted in the enumeration.
      We will see that, without loss of generality in the present proof, we will be able to assume that $W$ satisfies all conditions simultaneously.
      Consider a general element $e\in S'$. We write
      \begin{equation*}
        e=\sum_{i=1}^9 c_i e_i,
      \end{equation*}
      where $c_2=0$ unless $W=U''\sqcup V$, and similarly for $c_3$, $c_4, c_5, c_7,$ and $c_8$.
      Suppose further that $e\in S'\cap \cT^{(2)}(E)$.
      This imposes equations on the $c_i$ by setting to zero the coefficients of any given cubic composition of generators appearing in the $e_i$.
      Furthermore, we may, without loss of generality, make two ``fine-tuning'' assumptions: first, that $W$ is of the necessary form as specified in the definition of each $e_i$ (e.g. $W=U''\sqcup V$) and second, that $W'$ is both of the form $U\sqcup V'$ and of the form $U'\sqcup V$.
      Indeed, to pass to the ``non-fine-tuned'' situation one replaces a single equation of the form
      \begin{equation*}
        f_1(c_1,\ldots, c_9)+f_2(c_1,\ldots, c_9)=0,
      \end{equation*}
      with the set of two equations
      \begin{equation*}
        f_1(c_1,\ldots, c_9)=0, \quad f_2(c_1,\ldots, c_9)=0
      \end{equation*}
      (and repeats this process if necessary).
      In other words, we find that the ``non-fine-tuned'' solutions are a subspace of the ``fine-tuned'' solutions.
      Proceeding with the fine-tuned case, we obtain the six equations
      \begin{align*}
        c_1-c_6-c_9  & =0 \\
        c_2+c_5      & =0 \\
        c_3+c_8      & =0 \\
        c_4+c_7      & =0 \\
        -c_5+c_6-c_7 & =0 \\
        -c_4-c_8+c_9 & =0.
      \end{align*}
      Each of these equations represents one of the six planar trees with two unary vertices and one binary vertex, e.g., the first equation corresponds to the diagram in Figure \ref{fig: binunun}.
      \begin{figure}[b]
        \begin{subfigure}[b]{.49\linewidth}
          \centering
          \begin{forest}
            smalloperad
            [[
                [[][]]
              ]]
          \end{forest}
          \caption{$c_1-c_6-c_9=0$}\label{fig: binunun}
        \end{subfigure}
        \begin{subfigure}[b]{.49\linewidth}
          \centering
          \begin{forest}
            smalloperad
            [[
                [][[]]
              ]]
          \end{forest}
          \caption{$-c_4-c_8+c_9=0$}\label{fig: unbinun}
        \end{subfigure}
        \caption{The cubic trees corresponding to two of the equations for the $c_i$.}\label{fig: simplecubic}
      \end{figure}
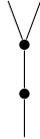
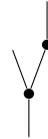
      Putting the coefficients $c_1,\ldots, c_9$ into a column vector, one obtains the following basis for the space of solutions to these equations:
      \begin{equation*}
        \left\{
        \begin{bmatrix}
          1 \\-1\\0\\0\\1\\1\\0\\0\\0
        \end{bmatrix},
        \begin{bmatrix}
          0 \\1\\-1\\-1\\-1\\0\\1\\1\\0
        \end{bmatrix},
        \begin{bmatrix}
          1 \\-1\\0\\1\\1\\0\\-1\\0\\1
        \end{bmatrix}
        \right\}.
      \end{equation*}
      One can verify directly that in all three cases, the corresponding sum $\sum_i c_i e_i$ is indeed quadratic in the generating operations.
      One therefore finds that the space $S'\cap \cT^{(2)}(E)$ is the space of all linear combinations
      \begin{equation}\label{eq: quadintersection}
        -c_1 (\ext[U\sqcup V][W]\circ \bin) -c_2 (\bin[U''][V]\circ_1 \ext[U][U'']) -c_3 (\bin[U][V'']\circ_2 \ext[V][V''] )
      \end{equation}
      where $c_1+c_2+c_3=0$.
      In the case that we have been considering, namely that all three summands in Equation \eqref{eq: quadintersection} have the same output color, one verifies directly the element in that equation belongs to $R_3$, using the equation $c_1+c_2+c_3=0$.
      As we have mentioned, we do not need to consider the more general case, for which the three operations do not have the same output color.
      \qedhere
  \end{description}
\end{proof}

\section{Proof of the Koszul property}\label{sec disj koszul}

In this section, we prove the Koszulity of the colored operad $\Disj$, and derive an explicit description of the Koszul dual $\Disj^\antishriek$.
This section can be omitted on a first reading, since its main function is to provide in Section~\ref{sec descr htpy prefac} an explicit description of $\hoDisj$ algebras.

Recall from item~\ref{it not operads} of Section~\ref{sec conventions} that we use the notation $q\PCosh$ to denote the operad obtained by projecting the relations of the quadratic-linear operad $\PCosh$ onto their quadratic parts.
Concretely, the operad $q\PCosh$ is generated by the unary operations $\ext$ and the binary operations $\bin$ subject to the same relations $R_2$ and $R_3$ (which are purely quadratic) and the quadratic version of $R_1$ given by $qR_1 = \bigoplus \RR \{ \ext[U][V]\circ \ext[V][W]\}$---that is, the composition of any two unary generators vanishes.

\begin{lemma}\label{lem: openkoszul}
  The quadratic colored operad $q\PCosh$ is Koszul.
\end{lemma}
\begin{proof}
  The colored operad $q\PCosh$ is a colored version of the dual numbers algebra. Its Koszul dual cooperad is the cofree colored cooperad on generators $\tau_{U}^V$ of degree $-1$ for each strict inclusion $U\subsetneq V$. The proof of the Koszulity of the dual numbers algebra applies equally well here.
\end{proof}

\begin{lemma}\label{lem: tenskoszul}
  The quadratic colored operad $\Tens$ is Koszul.
\end{lemma}
\begin{proof}
  We need to show that the Koszul complex $\Tens^{\antishriek}\circ_\kappa \Tens$ is acyclic.

  The operad $\Tens$ is generated by the binary operations $\bin$ subject to the relations $R_2$.
  For this reason, it resembles a colored version of the commutative operad.
  Our strategy is therefore to reduce the case under consideration to the proof of the Koszulity of the commutative operad.
  To this end, note that the space
  \begin{equation*}
    \Tens\binom{V}{U_1,\ldots, U_k}
  \end{equation*}
  is one-dimensional if the $U_i$ are pairwise disjoint and $V=U_1\sqcup \cdots \sqcup U_k$; the space is 0 otherwise.
  In other words, if one fixes the input colors of  an operation, the output color is determined by the inputs.
  When it is non-zero, the corresponding space of operations agrees with the one for the commutative operad.
  We will use variations of this basic observation repeatedly.

  Now, we  wish to understand the Koszul dual cooperad $\Tens^{\antishriek}$.
  Let us first consider the cofree cooperad $\cT^c(E(2)[1])$ on the binary generators $\bin$.
  The Koszul dual cooperad $\Tens^{\antishriek}$ is a sub-cooperad of $\cT^c(E(2)[1])$ which we will consider in a moment.
  Let us fix a collection $U_1,\ldots, U_k$ of input colors.
  A straightforward induction shows that
  \begin{equation*}
    \cT^c(E(2)[1])\binom{V}{U_1,\ldots, U_k}
  \end{equation*}
  is the zero vector space unless the $U_i$ are pairwise disjoint and $V=U_1\sqcup \cdots \sqcup U_k$.
  In the latter case, the vector space has a basis consisting of binary shuffle trees (the input colors determine the operations at the vertices and the colors on the edges and root of the tree).
  In other words,
  \begin{equation*}
    \cT^c(E(2)[1])\binom{V}{U_1,\ldots, U_k} = \cT^c(F[1])(k),
  \end{equation*}
  where $F$ is the trivial $\SS_2$-module concentrated in arity 2.
  To lie in $\Tens^{\antishriek}$, a linear combination of trees needs to have relations from $R_2$ on all pairs of adjacent vertices.
  Fixing pairwise disjoint input colors $U,V,W$, the space of relators
  \begin{equation*}
    R_2\binom{U\sqcup V\sqcup W}{U,V,W}
  \end{equation*}
  is naturally identified with the space of relations defining the single-colored operad governing commutative algebras.
  Hence, we obtain
  \begin{equation*}
    \Tens^{\antishriek}\binom{V}{U_1,\ldots, U_k} =
    \begin{cases}
      \Com^{\antishriek}(k)= \Lie^c\{1\}, & V=U_1\sqcup \cdots \sqcup U_k; \\
      0                                   & \text{else}.
    \end{cases}
  \end{equation*}
  Moreover, given $U_{11},\ldots, U_{1p_1},\ldots, U_{kp_k}$ pairwise disjoint, set $W_i= U_{i1}\sqcup\cdots \sqcup U_{ip_i}$ and $V=W_1\sqcup\cdots \sqcup W_k$.
  The only non-zero cocomposition map
  \begin{align*}
    \Tens^{\antishriek} \binom{V}{U_{11},U_{12},\ldots , U_{kp_k}}\to \Tens^{\antishriek}\binom{V}{W_1,\ldots, W_k} \otimes \bigotimes_{i=1}^k\Tens^{\antishriek}\binom{W_i}{U_{i1},\ldots, U_{ip_i}}
  \end{align*}
  is given by the cocomposition in $\Lie^c\{1\}$.

  Finally, let us consider the Koszul complex $\Tens^{\antishriek} \circ_\kappa \Tens$.
  Let us fix a set of input colors $U_1,\ldots, U_k$.
  Unless the $U_i$ are pairwise disjoint and $V=U_1\sqcup \cdots \sqcup U_k$, we have
  \begin{equation*}
    \left( \Tens^{\antishriek} \circ_\kappa \Tens\right) \binom{V}{U_1,\ldots, U_k}=0.
  \end{equation*}
  In the case that the Koszul complex is non-zero, it is precisely the Koszul complex for $\Com$ in arity $k$, hence is acyclic.
  %
\end{proof}

\begin{lemma}\label{lem: compprod}
  Consider the quadratic operad $\cT(E,qR)$ associated to the quadratic-linear operad $\Disj\cong \cT(E,R)$. The natural map
  \begin{equation*}
    \Psi:q\PCosh\circ \Tens \to \cT(E,qR)
  \end{equation*}
  of colored $\mathbb S$-modules is an isomorphism.
\end{lemma}
\begin{proof}
  Note the following: the colored operad $\cT(E,qR)$ is obtained via a rewriting rule
  \begin{equation*}
    \lambda: \Tens\circ_{(1)}q\PCosh \to q\PCosh\circ_{(1)}\Tens.
  \end{equation*}
  Indeed, the relation $R_3$ tells us precisely how to turn any composite in $\Tens\circ_{(1)}q\PCosh$ into a composite in $q\PCosh\circ_{(1)}\Tens$.
  Hence, we may write
  \begin{equation*}
    \cT(E,qR) = q\PCosh\vee_\lambda \Tens,
  \end{equation*}
  and the rewriting rule allows us to conclude that the map $\Psi$ is an arity-wise surjection.

  Let $q\Disj$ denote the colored operad whose underlying colored $\SS$-module is the same as $\Disj$, but such that we may only compose $\mu_{U_1,\ldots, U_k}^V$ with $\mu_{V_1,\ldots, V_{k'}}^{W}$ non-trivially, where $V = V_i$, if either $V=U_1\sqcup\cdots \sqcup U_k$ or $W=V_1\sqcup \cdots V_{i-1}\sqcup \cdots \sqcup V_{k'}$.
  It is manifest that, as a colored $\SS$-module,
  \begin{equation*}
    q\Disj\cong q\PCosh\circ \Tens.
  \end{equation*}
  Moreover, there is a natural map of operads
  \begin{equation*}
    \Psi':\cT(E,qR)\to q\Disj
  \end{equation*}
  constructed exactly as in the proof of Proposition~\ref{prop: gensrels}, and this map is an isomorphism for the same reasons as in the proof there.
  The composite map $\Psi'\circ \Psi$ is readily seen to be an isomorphism of $\SS$-modules; hence, $\Psi$ is a monomorphism, which completes the proof.
\end{proof}

\begin{theorem}\label{thm: disjkoszul}
  The operad $\Disj$ is Koszul.
\end{theorem}

\begin{proof}
  Since $\Disj\cong \cT(E,R)$ is a quadratic-linear operad, Koszulity of $\Disj$ is defined to be Koszulity of the associated quadratic operad $\cT(E,qR)$.
  By Lemma~\ref{lem: compprod}, $\cT(E,qR)$ can be written as a composition product via a rewriting rule.
  By the Diamond Lemma,
  Koszulity of $\cT(E,qR)$ follows from the Koszulity of the factors, which is the content of Lemmas \ref{lem: openkoszul} and \ref{lem: tenskoszul}.
\end{proof}

At this point, we have provided a resolution $\hoDisj = \Omega \Disj^\antishriek$ of the colored operad $\Disj$, and this resolution is smaller than the bar-cobar resolution.

We can use the preceding arguments to describe the Koszul dual $\cT(E,R)^\antishriek$.
Before we do this, however, let us establish a bit of notation.
Note that, because we have applied the diamond lemma to the operad $q\cT(E,R)$, by \autocite[Theorem~B.3]{millescmha}, we may write $q\cT(E,R)^\antishriek \cong \Tens^\antishriek \circ q\PCosh^\antishriek$.
Given disjoint subsets $U_{1s_1},\ldots, U_{ks_k}$, we let
$\mu^\antishriek_{U_{1s_1},\ldots, U_{ks_k}}$
denote the image of the $k$-ary operation under the map
\begin{equation*}
  \As^\antishriek(k)\to\Com^\antishriek(k)\cong \Tens^\antishriek\binom{U_{1s_1}\sqcup\dots\sqcup U_{ks_k}}{U_{1s_1},\dots,U_{ks_k}}.
\end{equation*}
(Note that the while the $\mu^\antishriek_{U_1,\ldots, U_k}$ generate the spaces of $k$-ary cooperations as $\SS$-modules, there are relations between these operations.)
Let $\extantishriek$ denote the cogenerator in $\cT^c(sE,s^2R)=\cT(E,R)^\antishriek$ corresponding to $\ext$.
For each $1\leq i\leq k$, given a chain $\cU_i = (U_{i1}\subsetneq U_{i2}\subsetneq\cdots \subsetneq U_{is_i})$ of inclusions, we let:
\begin{equation*}
  \iota^\antishriek_{\cU_i}\coloneqq \extantishriek[U_{i(s_{i}-1)}][U_{is_i}]\cocirc\cdots\cocirc \extantishriek[U_{i1}][U_{i2}],
\end{equation*}
where the notation $\cocirc$ denotes the formal ``composite'' in the cooperad $(q\PCosh)^\antishriek$ (if $s_i=1$, we understand $\iota^\antishriek_{\cU_i}$ to mean the identity).
Finally, we use $\cU$ to denote the full (ordered) collection $(\cU_1,\ldots, \cU_k)$, and define
\begin{equation*}
  \mu^\antishriek_\cU = (\mu^\antishriek_{U_{1s_1},\ldots, U_{ks_k}};\iota^\antishriek_{\cU_1},\ldots, \iota^\antishriek_{\cU_k});
\end{equation*}
given a $(s_1-1,\ldots, s_k-1)$ shuffle $\sigma$, we let $\sigma\cdot \cU$ denote the unique chain of inclusions starting with $U_{11}\sqcup \cdots \sqcup U_{k1}$ and ending with $U_{1s_1}\sqcup \cdots\sqcup U_{ks_k}$ corresponding to $\sigma$ (see Remark~\ref{rem shufflecomposition} for an example).

\begin{lemma}
  The cooperadic (co)distributive law defining the cooperadic structure on
  \begin{equation*}
    (q\cT(E,R))^\antishriek \cong  \Tens^\antishriek\circ \PCosh^\antishriek
  \end{equation*}
  is given by the map
  \begin{align}
    \Lambda^c & :  \Tens^\antishriek \circ \PCosh^\antishriek \to \PCosh^\antishriek\circ \Tens^\antishriek\nonumber \\
    \Lambda^c & (\mu^\antishriek_\cU)=\sum_{\sigma\in \Sh(s_1-1,\ldots,s_k-1)}(-1)^{(k-1)\sum(s_i-1)}\sgn(\sigma)\iota^\antishriek_{\sigma\cdot\cU}\,\cocirc \,\mu^\antishriek_{U_{11},\ldots, U_{k1}}.\label{eq: codistr}
  \end{align}
  The differential in the Koszul dual is given by
  \begin{align}
    d\mu^\antishriek_\cU & = \sum_{i=1}^k (-1)^{(k-1)+(\sum_{j=1}^{i-1}(s_j-1))}(\mu^\antishriek_{U_{1s_1},\ldots, U_{ks_k}}; \iota^\antishriek_{\cU_1},\ldots ,d\iota^\antishriek_{\cU_i},\ldots, \iota^\antishriek_{\cU_k})
    \nonumber \\
                         & \coloneqq \sum_{i=1}^k (-1)^{(k-1)+(\sum_{j=1}^{i-1}(s_j-1))} \mu^\antishriek_{d_i\cU}
    \label{eq diff Koszul dual}
  \end{align}
  here,
  \begin{equation*}
    d\iota^\antishriek_{\cU_i} = \sum_{j=1}^{s_i-2}(-1)^{j-1}\iota^\antishriek_{\cU_i,\hat j},
  \end{equation*}
  where $\iota^\antishriek_{\cU_i,\hat j}$ is the result of replacing $\extantishriek[U_{i(j+1)}][U_{i(j+2)}]\cocirc\extantishriek[U_{ij}][U_{i(j+1)}]$ with $\extantishriek[U_{ij}][U_{i(j+2)}]$ in $\iota^\antishriek_{\cU_i}$.
\end{lemma}

\begin{remark}\label{rem shufflecomposition}
  To clarify the meaning of the symbol $\iota^\antishriek_{\sigma\cdot \cU}$, let us describe it in a simple example: let $k=2$, $s_1=3$, $s_2=2$, so that $\cU$ consists of a chain $U_{11}\subsetneq U_{12}\subsetneq U_{13}$ and a disjoint chain $U_{21}\subsetneq U_{22}$. Let $\sigma(1) = 1$, $\sigma(2)=3$, $\sigma(3)=2$ (that is, $\sigma\in \mathrm{Sh}(2,1)$). Then,
  \begin{equation*}
    \iota^\antishriek_{\sigma\cdot \cU} =
    \extantishriek[U_{12}\sqcup U_{22}][U_{13}\sqcup U_{22}]
    \cocirc\extantishriek[U_{12}\sqcup U_{21}][U_{12}\sqcup U_{22}]
    \cocirc\extantishriek[U_{11}\sqcup U_{21}][U_{12}\sqcup U_{21}].
    \qedhere
  \end{equation*}
\end{remark}

\begin{proof}
  The claim about the differential follows directly from the definitions.
  The most subtle thing in the proof thereof is the Koszul sign rule.

  We have already established the Koszulity of $q\Disj$ using the diamond lemma.
  It follows from \autocite[Proposition~B.2]{millescmha} that the Koszul dual $q\Disj^\antishriek$ can be described by a codistributive law, which is itself induced from the rewriting rule $\lambda$ used in the proof of Lemma \ref{lem: compprod}
  It remains to make explicit the consequences of this codistributive law.
  To this end, we make heavy reference to the proof of \autocite[Lemma~2.2]{millescmha}, where a similar proof is undertaken for an operad generated by:
  \begin{itemize}
    \item a binary operation of degree +1,
    \item a unary operation of degree 0
  \end{itemize}
  subject to the relations
  \begin{itemize}
    \item the binary operation satisfies the Jacobi relation
    \item the unary operation squares to 0
    \item the binary and unary operations satisfy the same rewriting rule as in $q\Disj$.
  \end{itemize}
  The present case differs from the case studied in \autocite{millescmha} in the following three ways:
  \begin{enumerate}
    \item\label{diff-bm-1} there are additional open subsets of $M$ labeling all operations
    \item\label{diff-bm-2} the binary operations form an analogue of the commutative operad rather than the Lie operad, and
    \item\label{diff-bm-3} the binary operations have degree 0 rather than degree +1.
  \end{enumerate}
  Item~\ref{diff-bm-1} presents only the challenge of notational complexity; item~\ref{diff-bm-2} is immaterial; and item~\ref{diff-bm-3} is responsible for the sign $(-1)^{(k-1)\sum(s_i-1)}$ appearing in Equation~\eqref{eq: codistr}.
  The proof of Equation~\eqref{eq: codistr} is, as in \autocite{millescmha}, by a double induction on the number of binary cogenerators and the number of unary cogenerators.
  Since the proof is similar to that case, we refer the reader thither for all the details; here, we will comment on the necessary modifications in the present case and on a few points which are omitted in the proof in \autocite{millescmha} but which we found enlightening to understand.
  The first difference lies in establishing the base cases for the induction.
  We need to establish the base case that $k=2$, $s_1+s_2=2$.
  This corresponds to tracing
  \begin{equation*}
    (\binantishriek;\extantishriek[U'][U],\id)
  \end{equation*}
  (and the composition in the other factor) through the codistributive law.
  The codistributive law is the composite
  \begin{equation*}
    \Tens^\antishriek  \circ (q\PCosh)^\antishriek \to \cT^c(E[1],R[2])\to (q\PCosh)^\antishriek \circ \Tens^\antishriek,
  \end{equation*}
  where the second map is the projection onto trees where all the binary vertices are above the unary ones, and the first map is the inverse of the analogous map for the binary vertices below the unary ones.
  The image of $(\binantishriek;\extantishriek[U'][U],\id)$ under the first map is
  \begin{equation*}
    (\binantishriek;\extantishriek[U'][U],\id)- (\extantishriek[U'\sqcup V][U\sqcup V];\binantishriek[U'][V]).
  \end{equation*}
  Hence,
  \begin{equation*}
    \Lambda^c(\binantishriek;\extantishriek[U'][U],\id) = -(\extantishriek[U'\sqcup V][U\sqcup V];\binantishriek[U'][V]),
  \end{equation*}
  which establishes the base case, and accounts for the sign $(-1)^{(k-1)\sum(s_i-1)}$ appearing in Equation \eqref{eq: codistr}

  The rest of the proof proceeds exactly as in \autocite{millescmha}.
  We make two elucidations thereof.
  The first is that the Koszul sign rules \autocite[Section~5.1.8]{LodayVallette2012} appearing in the associator for the composition product $\circ$ play an important role in the proof.
  In Bellier-Millès's version, these appear as the $\epsilon_{k_j',k''_j}$ signs.
  Those signs appear in our proof as well; in our case, there are further signs as a result of the fact that both cogenerators have odd degree.
  Second, Bellier-Millès omits the induction on the number of binary cogenerators.
  Relevant to that proof is the fact that given any shuffle $\sigma \in \Sh(j_1,\ldots, j_k)$ and any partition $k=\ell_1+\cdots +\ell_n$, $\sigma$ can be written uniquely in the form:
  \begin{equation*}
    \tau \cdot \sigma_1\cdot \cdots \cdot \sigma_k,
  \end{equation*}
  where $\sigma_i \in \Sh(j_{\ell_{i-1}+1},\ldots,j_{\ell_{i}})$ and $\tau\in \Sh(\sum_{i=1}^{\ell_1}j_i,\ldots, \sum_{i=\ell_{n-1}+1}^{\ell_n}j_i)$.
\end{proof}

Now that we have made explicit the codistributive law which defines the cooperad structure on $\Disj^\antishriek$, we can also make the cooperad structure itself more explicit.
That is the object of the following lemma.
Its statement is long because of the combinatorics of the trees involved.
However, we will provide an example that will illustrate this lemma graphically.
\begin{lemma}\label{lem: infdecomp}
  Let $\cU = (\cU_1, \dots, \cU_k)$ be as above.
  Under the isomorphism of $\SS$-modules
  \begin{equation*}
    \Tens^\antishriek \circ \PCosh^\antishriek \cong \Disj^\antishriek,
  \end{equation*}
  the infinitesimal decomposition map $\Delta^{(1)}_{\Lambda^c}$ is given by the equation
  \begin{equation}
    \Delta^{(1)}_{\Lambda^c}(\mu^\antishriek_{\cU}) = \sum_{(p,q,j,s'_i, s''_i, \sigma)} \sgn(\sigma)(-1)^{\epsilon(s'_i,s''_i,j,p,q)}\mu^\antishriek_{\cU'}\circ_j\mu^\antishriek_{\cU''},
  \end{equation}
  where the sum ranges over:
  \begin{itemize}
    \item positive indices $(p,q)$ such that $p+q=k+1$;
    \item indices $j$ from $1$ to $p$;
    \item indices $s'_i$, $s''_i$ satisfying $s_i - 1 = s'_i + s''_i - 2$, $s''_i = 1$ if either $i \leq j-1$ or $i > j+q-1$;
    \item shuffles $\sigma \in \Sh(s'_1-1,\ldots, s'_q-1)$;
  \end{itemize}
  and where we let:
  \begin{itemize}
    \item $\cU'' = (\cU''_1,\ldots, \cU''_q)$ with $\cU''_l = U_{(j+l-1)s_1} \subsetneq \cdots\subsetneq U_{(j+l-1)s''_l}$;
    \item $\cU'_1 = U_{11}\subsetneq \cdots\subsetneq U_{1s_1} , \ldots , \cU'_{j-1}=U_{(j-1)1}\subsetneq\cdots \subsetneq U_{(j-1)s_{j-1}}$;
    \item $\cU'_j = \sigma\cdot (U_{j(s''_j+1)}\subsetneq \cdots \subsetneq U_{js_j}, \ldots, U_{(j+q-1)(s''_{j+q-1}+1)}\subsetneq \cdots \subsetneq U_{(j+q-1)s_{j+q-1}})$;
    \item $\cU'_{j+1} = U_{(q+j)1}\subsetneq \cdots \subsetneq U_{(q+j)s_{q+j}},\quad \ldots, \quad \cU'_p= U_{k1}\subsetneq \cdots \subsetneq U_{ks_k}$;
    \item $\cU' =(\cU'_1,\ldots, \cU'_p)$;
    \item $\epsilon(s'_i,s''_i,j,p,q)=(q+1)(p-j)+(q-1)\Bigl(\sum_{i=1}^{k}(s'_i-1)\Bigr)+\sum_{i=1}^k \Bigl((s''_i-1)\sum_{\ell=i+1}^k (s'_\ell-1)\Bigr)$.
  \end{itemize}

  The full decomposition map is given by the equation

  \begin{equation}\label{eq: fulldecomp}
    \Delta_{\Lambda^c}(\mu^\antishriek_\cU)
    = \sum_{\substack{(p, q_1,\ldots, q_p, \\ s'_i,s''_i,\sigma_1,\ldots, \sigma_p)}}
    \Bigl(\prod_{i=1}^p \sgn(\sigma_i) \Bigr) (-1)^{\delta(s',s'',q,p)} (\mu^\antishriek_{\cU'}; \mu^\antishriek_{{}_1\cU''},\ldots, \mu^\antishriek_{{}_p\cU''} ),
  \end{equation}
  where the sum is over:
  \begin{itemize}
    \item positive indices $(p,q_1,\ldots, q_p)$ such that $q_1+\cdots +q_p = k$;
    \item indices $s'_i, s''_i$ satisfying $s_i-1 = s'_i+s''_i-2$ for $1\leq i\leq k$;
    \item shuffles $\sigma_j \in \Sh(s'_{q_1+\cdots+q_{j-1}+1}-1,\ldots, s'_{q_1+\cdots+q_j}-1)$;
  \end{itemize}
  and where we let
  \begin{itemize}
    \item $\mu_{j\cU''}$ ($1\leq j\leq p$) and $\cU'$ be defined analogously to the infinitesimal case, and
    \item $\delta(s',s'',q,p)=\sum_{j=1}^p(q_j+1)(p-j)+ \Bigl(\sum_{j=1}^p(q_j-1)\Bigr) \Bigl(\sum_{i=1}^{k}(s'_i-1)\Bigr)+\sum_{i<\ell} \Bigl((s''_i-1)(s'_\ell-1)\Bigr)$.
  \end{itemize}
\end{lemma}

\begin{proof}
  The proof is very similar to the proof of \autocite[Proposition~2.3]{millescmha}.
  In $\epsilon$ and $\delta$, the term $(q+1)(p-j)$ is the sign that appears in the decomposition product for $\As^\antishriek$ \autocite[Lemma~9.1.2]{LodayVallette2012}.
  The term $(q-1)\sum(s'-1)$ appears because of the corresponding term in the codistributive law.
  The last term $\sum\left((s''-1)\sum(s'-1)\right)$ appears for the same reasons as it does in \autocite{millescmha}, i.e. from the sign rules arising in the associator for the composition product $\circ$ of $\SS$-modules.
\end{proof}

\section{Description of homotopy prefactorization algebras}\label{sec descr htpy prefac}

In this section, we make explicit the structure present in a $\hoDisj$ algebra, and we describe a homotopy transfer theorem for $\hoDisj$ algebras.

\begin{proposition}\label{prop: hodisjexplicit}
  A $\hoDisj$ algebra is a collection of spaces $\cA(U)$, one for every open subset $U\subseteq M$, equipped with maps:
  \begin{equation*}
    \mu_{\cU}: \cA(U_{11})\otimes \cdots \otimes \cA(U_{k1})\to \cA(U_{1s_1}\sqcup \cdots \sqcup U_{ks_k})
  \end{equation*}
  for every collection $\cU = \bigl( U_{11} \subset \dots \subset U_{1s_1}, \dots, U_{k1} \subset \dots \subset U_{k s_k}\bigr)$ as in Lemma~\ref{lem: compprod}, such that:
  \begin{enumerate}[nosep]
    \item The $\mu_\cU$ have degree $2-k-\sum_i (s_i-1)$.
    \item\label{it mu shuf}
          The $\mu_\cU$ vanish on graded sums of shuffle permutations, i.e.
          \begin{equation*}
            \sum_{\sigma\in \Sh(\ell_1,\ldots, \ell_n)\subset \SS_k} \delta_\sigma\,\mu_{\cU\cdot\sigma} \circ \sigma =0,
          \end{equation*}
          where:
          \begin{itemize}[nosep]
            \item $\sigma$ is understood as a map
                  \begin{equation*}
                    \cA(U_{11})\otimes \cdots \otimes \cA(U_{k1})\to \cA(U_{\sigma^{-1}(1)1})\otimes \cdots \otimes \cA(U_{\sigma^{-1}(k)1}),
                  \end{equation*}
            \item The symbol $\cU\cdot\sigma$ has the meaning
                  \begin{equation*}
                    \cU\cdot\sigma = (\cU_{\sigma^{-1}(1)},\ldots, \cU_{\sigma^{-1}(k)}),
                  \end{equation*}
            \item The sign $\delta_\sigma$ is determined by the sign incurred from the obvious action of $\sigma$ on
                  \begin{equation*}
                    \Lambda^{\mathrm{top}}\RR^{k}\otimes \Lambda^{\mathrm{top}}\RR^{s_1-1}\otimes \cdots \otimes \Lambda^{\mathrm{top}}\RR^{s_k-1}
                  \end{equation*}
          \end{itemize}
          \item\label{it diff mu cU}
          The $\mu_\cU$ satisfy the following relations:
          \begin{multline}
            d(\mu_\cU)= \sum_{i=1}^k (-1)^{k+\sum_{j=1}^{i-1}(s_j-1)} \mu_{d_i\cU} \\
            +\sum_{\substack{(p, q_1,\ldots, q_p, \\ s'_i,s''_i,\sigma_1,\ldots, \sigma_p)}} \sgn(\sigma)(-1)^{p-1+\sum(s'_i-1)}(-1)^{\epsilon(s'_i,s''_i,j,p,q)}\mu_{\cU'}\circ_j\mu_{\cU''},
            \label{eq diff mu}
          \end{multline}
          where the second sum ranges over:
          \begin{itemize}[nosep]
            \item positive indices $(p,q_1,\ldots, q_p)$ such that $q_1+\cdots +q_p = k$;
            \item indices $s'_i, s''_i$ satisfying $s_i-1 = s'_i+s''_i-2$ for $1\leq i\leq k$;
            \item shuffles $\sigma_j \in \Sh(s'_{q_1+\cdots+q_{j-1}+1}-1,\ldots, s'_{q_1+\cdots+q_j}-1)$.
            \item the symbols $\cU'$, $\cU''$, and $\epsilon(s_i',s_i'',j,p,q)$ are defined in the statement of Lemma~\ref{lem: infdecomp}.
          \end{itemize}
  \end{enumerate}
\end{proposition}

\begin{proof}
  A $\hoDisj$-algebra is, by definition, an algebra over the colored operad $\Omega \Disj^\antishriek$, which is semi-free on the shift of the reduced cooperad $\overline{\Disj}^\antishriek$.
  Hence, a $\hoDisj$-algebra has one operation for every non-identity cooperation in $\Disj^\antishriek$.
  The operad $\Omega\Disj^\antishriek$ has a differential which is the sum of two terms: one induced from the differential on $\Disj^\antishriek$ (Equation~\eqref{eq diff Koszul dual}), and one induced from the cooperadic structure (Lemma~\ref{lem: infdecomp}).
  These terms correspond, respectively, to the two separate sums in Condition~\ref{it diff mu cU} of the Proposition.

  The only thing which has not been spelled out in the preceding propositions is Condition~\ref{it mu shuf}.
  To establish this condition, we note that because $\Tens$ resembles a commutative operad, the sum over shuffles should be unsurprising.
  The main differences from the usual commutative operad are the presence of colors and of the extra composition factor $\PCosh$ in $\Disj$.
  We deal with the first issue by permuting also the colors (i.e. by introducing the $\cU\cdot \sigma$).
  We deal with the second issue by introducing the extra $s_i$-dependent signs in $\delta_\sigma$.
\end{proof}

Now that we have made explicit what it is to have an algebra over the operad $\hoDisj$, we can also say what we mean by an infinity-morphism of $\hoDisj$-algebras.

\begin{definition}
  Let $\cA$ and $\cB$ be algebras over the operad $\hoDisj$.
  By the Rosetta Stone of the theory of homotopy algebras \autocite[Theorem~10.1.13]{LodayVallette2012}, these are given by codifferentials $D_\cA$, $D_\cB$ on the cofree $\Disj^\antishriek$-algebras $\Disj^\antishriek(\cA), \Disj^\antishriek(\cB)$, respectively.
  An \textbf{infinity-morphism} of $\hoDisj$-algebras $\cA\rightsquigarrow \cB$ is a map of dg-$\Disj^\antishriek$-coalgebras
  \begin{equation*}
    (\Disj^\antishriek(\cA),D_\cA)\to(\Disj^\antishriek(\cB),D_\cB)
  \end{equation*}
  of the underlying semi-cofree $\Disj^\antishriek$-coalgebras.
\end{definition}

The above definition is compact, but not so useful in making explicit what one needs to check in practice to guarantee that one has an infinity-morphism.
The below proposition aims to remedy this situation:

\begin{proposition}
  Let $\cA$ and $\cB$ be algebras over the operad $\hoDisj$, with operations $\mu_\cU^\cA$ and $\mu_\cU^\cB$, respectively,
  An infinity-morphism $\cA \rightsquigarrow \cB$ is given by a collection of maps
  \begin{equation*}
    f_{\cU}: \cA(U_{11})\otimes \cdots \otimes \cA(U_{k1})\to \cB(U_{1s_1}\sqcup\cdots\sqcup \cB(U_{1s_k}))
  \end{equation*}
  of degree $1-k- \sum_i (s_i-1)$ satisfying the symmetry property 2 in Proposition~\ref{prop: hodisjexplicit} and the relation
  \begin{align}\label{eq: inftymorphism}
    d & (f_\cU) = \sum_{i=1}^k (-1)^{k+1+\sum_{j=1}^{i-1}s_j-1} f_{d_i\cU} \nonumber \\
    + & \sum_{(p,q,j,s_i,s''_i,\sigma)} \sgn(\sigma)(-1)^{p-1+\sum(s'_i-1)+\epsilon(s'_i,s''_i,j,p,q)} f_{\cU'}\circ_j \mu^A_{\cU''}\nonumber \\
      & - \sum_{(p,s'_i,s''_i,q_j,\sigma_j)} \Bigl( \prod_{j=1}^p \sgn(\sigma_j) \Bigr) (-1)^{\delta(s'_i,s''_i,q_j,p)} (\mu^B_{\cU'}; f_{{}_1\cU''},\ldots f_{{}_p\cU''})
  \end{align}
\end{proposition}
\begin{proof}
  This proposition follows from the explicit description of the cooperad $\Disj^\antishriek$ and the discussion of cylinder objects in \autocite{FresseCobar} (cf. Figure 10 therein).
\end{proof}

The following statement follows from the general facts of Koszul theory (cf.~\autocite{LodayVallette2012}, Theorem 10.3.1)

\begin{proposition}\label{prop HTT}
  Let $\cA$ be a $\hoDisj$ algebra and suppose given, for every open set $U\subset M$, a deformation retraction
  \[
    \ourDR{\cB(U)}{\cA(U)}{i_U}{p_U}{h_U};
  \]
  then, there exists a $\hoDisj$ structure on the collection $\{\cB(U)\}$ such that the maps $i_U$ extend to an $\infty$-morphism~$\cB\rightsquigarrow \cA$.
\end{proposition}

\begin{remark}\label{rmk koszul self dual}
  The description above dualizes to a description of complete algebras over $\Disj^! = (q\Disj^!, d)$.
  (Note that we must use complete algebras because the space of generators is not finite dimensional.)
  A complete $\Disj^!$-algebra is given by a collection of dg-spaces $\cC(U)$, one for every open subset $U \subseteq M$, equipped with maps:
  \begin{align*}
    \extshriek : \cC(U) & \to \cC(V), & \text{for } U \subsetneq V; \\
    \binshriek : \cC(U) \otimes \cC(V) &\to \cC(U \sqcup V), & \text{for } U \cap V = \emptyset;
  \end{align*}
  of respective degrees $\deg(\extshriek) = 0$ and $\deg(\binshriek) = 1$.
  The binary operations are antisymmetric and satisfy a relation similar to that of shifted Lie algebras: for disjoint open sets $U, V, W \subseteq M$, we have
  \begin{equation*}
    \binshriek[U \sqcup V][W](\binshriek[U][V](x,y), z) \pm \binshriek[V \sqcup W][U](\binshriek[V][W](y,z), x) \pm \binshriek[W \sqcup U][V](\binshriek[W][U](z,x), y) = 0.
  \end{equation*}
  The unary operations satisfy compatibility relation with the differential: for $U \subsetneq W$, we have
  \begin{equation*}
    d \bigl( \extshriek[U][W](x) \bigr) = \extshriek[U][W]\bigl( d(x) \bigr) + \sum_{U \subsetneq V \subsetneq W} \extshriek[V][W] \bigl( \extshriek[U][V](x) \bigr).
  \end{equation*}
  Finally, the binary and unary operations satisfy a compatibility relation that essentially make unary operations into derivations of Lie algebras.
  Note that per this description, the operad $\Disj$ is not Koszul self-dual, unlike e.g., the $E_n$ operads.
\end{remark}

\section{Examples}\label{sec examples}

In this section, we apply the general theory of the preceding sections to some examples.
First, we discuss what it means to give a $\hoDisj$-algebra on some simple topological spaces.
Next, we discuss a result concerning a factorization algebra on $\RR$, extending a result of \textcite{CG1}.

\subsection{Preliminary: (\texorpdfstring{$\infty$}{infinity}-)modules over an algebra over an operad}

To make explicit the descriptions of homotopy prefactorization algebras on several finite topological spaces, we will need to recall a few background notions.

\begin{definition}[See e.g. {\autocite[Section~12.3.1]{LodayVallette2012}}]
  Let $\cN$ be a symmetric sequence and $A, M$ be dg-modules.
  The linearized composition product is defined by:
  \begin{equation*}
    \cN \circ (A; M) \coloneqq \bigoplus_{n \geq 0} \cN(n) \otimes_{\Sigma_n} \Bigl( \bigoplus_{i=1}^n A^{\otimes i-1} \otimes M \otimes A^{\otimes n-i} \Bigr).
  \end{equation*}
  Suppose now that $\cP$ is a dg-operad and $A$ is a $\cP$-algebra.
  An $A$-$\cP$-module (or simply $A$-module if $\cP$ is obvious from the context) is an object $M$ equipped with a map
  \begin{equation*}
    \gamma_M : \cP \circ (A; M) \to M
  \end{equation*}
  making the obvious diagrams commute.
\end{definition}

\begin{example}
  Any $\cP$-algebra $A$ is canonically a $\cP$-module over itself.
\end{example}

\begin{example}\label{ex infty module as}
  Let $\cP_\infty = \Omega\As^\antishriek$ be the $A_\infty$-operad and $A$ be an $A_\infty$-algebra.
  Then an $A$-module is a dg-module $M$ equipped with maps
  \begin{align*}
    \mu_{k, i} : A^{\otimes i - 1} \otimes M \otimes A^{k-i} & \to M[k-2], & \text{for } k \geq 2 \text{ and } 1 \leq i \leq k.
  \end{align*}
  These maps satisfy relations such that
  \begin{align*}
    \mu_{2,1} : M \otimes A & \to M,             & \mu_{2,2} : A \otimes M & \to M, \\
    m \otimes a             & \mapsto m \cdot a, & a \otimes m             & \mapsto a \cdot m,
  \end{align*}
  endow $M$ with a structure of an $A$-module up to homotopy.
  For example, if we write $a \otimes b \mapsto a * b$ for the binary operation in $A$, then we have the following relations (see Figure~\ref{fig infty module as} for a graphical representation):
  \begin{align*}
    (\partial \mu_{3,1})(m, a, b) & = m \cdot (a * b) - (m \cdot a) \cdot b, \\
    (\partial \mu_{3,2})(a, m, b) & = a \cdot (m \cdot b) - (a \cdot m) \cdot b, \\
    (\partial \mu_{3,3})(a, b, m) & = (a * b) \cdot m - a \cdot (b \cdot m).
    \qedhere
  \end{align*}
\end{example}

\begin{figure}[htbp]
  \centering
  \begin{equation*}
    \begin{forest}
      smalloperad
      [[,mod,baseline [1,mod][2][3]]]
    \end{forest}
    \xmapsto{\partial}
    \begin{forest}
      smalloperad
      [[,mod,baseline [1,mod][[2][3]]]]
    \end{forest}
    -
    \begin{forest}
      smalloperad
      [[,mod,baseline [,mod [1,mod][2]][3]]]
    \end{forest}
  \end{equation*}
  \caption{The first relation of Example~\ref{ex infty module as} illustrated by trees. The solid black edges are colored by $A$, the dashed red colored by $M$.}\label{fig infty module as}
\end{figure}
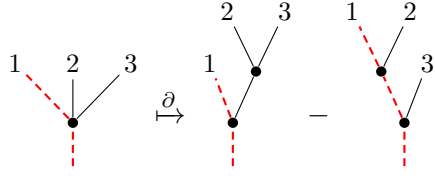

This notion could have been defined using the moperad~\autocite[Definition 9]{Willwacher2016} given by the shift $\cP(\,\_+1) = \{\cP(r+1)\}_{r \geq 0}$.
While we couldn't specifically find the next definition in the existing literature, it follows directly from the moperadic description (see \autocite[Section~10.2.2]{LodayVallette2012} for the unicolored case).

\begin{propdef}
  Let $\cP_{\infty} = \Omega \cC$ be the cobar construction of a dg-cooperad, let $\kappa : \cP_\infty \to \cC$ be the canonical Koszul twisting morphism, let $A,A'$ be $\cP_\infty$-algebras, and let $M, M'$ be modules over these algebras.
  An $\infty$-morphism $(f,g) : (A,M) \leadsto (A',M')$ is the data of:
  \begin{itemize}
    \item An $\infty$-morphism of $\cP_\infty$-algebras $f : A \leadsto A'$, i.e., a morphism of dg-$\cC$-coalgebras between the bar constructions $B_\kappa A = \bigl(\cC(A), d_\kappa\bigr) \to B_\kappa A' = \bigl(\cC(A'), d_\kappa\bigr)$.
          Such a morphism is uniquely determined by a map $f : \cC \circ A \to A'$ satisfying several compatibility relations.
    \item A map $g : \cC \circ (A; M) \to M'$ which is compatible with $f$ and the algebra/module structures.
          This map has to satisfy the following relations.
          Suppose that
          \begin{equation*}
            \xi = x(a_1, \dots, a_{i-1}, m, a_{i+1}, \dots, a_r) \in \cC(A; M)
          \end{equation*}
          is an element in the linearized composition product.
          Let $\Delta_{(1)}(x) = \sum_\alpha x'_\alpha \circ_{j_\alpha} x''_\alpha$ be an expression of the infinitesimal decomposition product applied to $x$, and let $\Delta(x) = \sum_\beta y'_\beta(y''_{\beta, 1}, \dots, y''_{\beta, k_\beta})$ be one for the total decomposition product.
          Then one has (up to signs that we will not write down here):
          \begin{multline*}
            (\partial g)(\xi) = \sum_\alpha
            \begin{cases}
              \pm (g(x'_\alpha) \circ_{j_\alpha} f(x''_\alpha))(a_1, \dots, m, \dots, a_r),
               & \text{if } j_\alpha \leq m \leq j_\alpha + |x''_\alpha| \\
              \pm (g(x'_\alpha) \circ_{j_\alpha} g(x''_\alpha))(a_1, \dots, m, \dots, a_r),
               & \text{otherwise};
            \end{cases}
            \\
            - \sum_\beta \pm (g(y'_\beta))(f(y''_{\beta,1}), \dots, g(y''_{\beta,i}), \dots, f(y''_{\beta,k_\beta}))(a_1, \dots, m, \dots, a_r).
          \end{multline*}
  \end{itemize}
\end{propdef}

\begin{example}
  Let us use the notation of Example~\ref{ex infty module as}.
  Let $A, A'$ be $A_\infty$-algebras, and $M, M'$ be an $A$-module (resp.\ $A'$-module).
  Then an $\infty$-morphism $(f,g) : (A, M) \leadsto (A', M')$ is the data of an $\infty$-morphism $f = \{f_k : A^{\otimes k} \to A'[k-1] \}_{k \geq 1}$ of $A_\infty$-algebras and a collection of maps
  \begin{align*}
    g_{k,i} : A^{\otimes i - 1} \otimes M \otimes A^{k-i} & \to M'[1-k], & \text{for } k \geq 1 \text{ and } 1 \leq i \leq k.
  \end{align*}
  These maps have to satisfy various compatibility relations with the $f_k$ and the differential, such that $g_{1,1} : M \to M'$ is a morphism of bimodules up to homotopy.
  For example, we have (using the notation of Example~\ref{ex infty module as}):
  \begin{align*}
    (\partial g_{2,1})(m, a) & = g_{1,1}(m) \cdot f_1(a) - g_{1,1}(m \cdot a), \\
    (\partial g_{1,2})(a, m) & = f_1(a) \cdot g_{1,1}(m) - g_{1,1}(a \cdot m).
    \qedhere
  \end{align*}
\end{example}

\subsection{Homotopy prefactorization algebras on a few finite topological spaces}\label{sec finite}

Let us now describe explicitly $\hoDisj$-algebras on a few finite topological spaces.

\begin{remark}
  While prefactorization algebras on finite spaces appear infrequently, this section is not just an exercise in abstraction.
  Given a manifold $X$ and a $\hoDisj_X$-algebra $\cA$, the structure maps associated to the sequence of inclusions $\emptyset \subset U \subset X$ for an open subset $U$ satisfy relations analogous to the ones in $\hoDisj_S$, where $S$ is the Sierpiński space of Lemma~\ref{lem sierp}.
  One can get a sense of the relationship by considering the quotient map $\pi : X \to S$ that collapses $U$ to $\{1\}$ and $X \setminus U$ to $\{2\}$, and pushing forward structure maps along $\pi$.
  More complex configurations of open sets can be described in a similar way using larger finite topological spaces.
\end{remark}

\begin{lemma}\label{lem hodisj empty}
  A $\hoDisj$-algebra on $\emptyset$ is a $C_\infty$-algebra.
  An infinity-morphism of $\hoDisj$-algebras on $\emptyset$ is an infinity-morphism of $C_\infty$-algebras.
\end{lemma}

\begin{proof}
  This is almost immediate.
  The generators of $\hoDisj$ are all of the form:
  \begin{equation*}
    m_k \coloneqq \mu^{\antishriek}_{\underbrace{\emptyset, \dots, \emptyset}_{k \text{ times}}}.
  \end{equation*}
  These correspond to the usual generators of the operad $C_{\infty}$.
  They vanish on signed shuffles and their differential (Equation~\eqref{eq diff mu}) is exactly the one in $C_\infty$.
\end{proof}

\begin{remark}\label{rem a empty}
  By restriction, if $\cA$ is a homotopy prefactorization algebra on any space $M$, then $\cA(\emptyset)$ is naturally endowed with a $C_\infty$-structure.
  It is common to normalize prefactorization algebras $\cA$ by requiring that $\cA(\emptyset)$ is the ground field.
\end{remark}

Let us now deal with the next simplest case, that of a singleton.

\begin{definition}
  Let $\cP_{\infty} = \Omega \cC$ be the cobar construction on a cooperad, and let $A$ be a $\cP_{\infty}$-algebra.
  A \textbf{pointed} $A$-module is an $A$-$\cP_\infty$-module $M$ equipped with an $\infty$-module morphism $A \leadsto M$.
\end{definition}

\begin{lemma}
  Let $X = \{1\}$ be a singleton (or more generally, a nonempty indiscrete space).
  A $\hoDisj$-algebra on $X$ is a triple $\cA = (A, M, \eta_M)$ where:
  \begin{itemize}[nosep]
    \item $\cA(\emptyset) = A$ is a $C_\infty$-algebra;
    \item $\cA(X) = M$ is a $C_\infty$-$A$-module;
    \item $\eta_M = \cA(\emptyset \subset X)$ is a pointing $(\id_A, \eta_M) : (A,A) \leadsto (A,M)$.
  \end{itemize}
  An infinity-morphism of $\hoDisj$-algebras on $X$ is an $\infty$-$C_\infty$-module morphism $(A, M) \leadsto (A', M')$ that commutes with the pointings.
\end{lemma}

\begin{proof}
  Let $\cA$ be a $\hoDisj$-algebra on $X$ and let $A = \cA(\emptyset)$, $M = \cA(X)$.
  By Remark~\ref{rem a empty}, the $\hoDisj$-structure on $\cA$ endows $A$ with a $C_\infty$-algebra structure, using the generators $\mu^{\antishriek}_{\emptyset, \dots, \emptyset}$.
  The remaining generators are of two kinds:
  \begin{enumerate}
    \item The generators $m_{k,i} \coloneqq \mu^{\antishriek}_{\cU[k, i]}$, where:
          \begin{equation*}
            \cU[k, i] = \underbrace{\bigl( \emptyset, \dots, \emptyset, \overbrace{X}^{i\text{th position}}, \emptyset, \dots, \emptyset \bigr)}_{k \text{ terms}}.
          \end{equation*}
          These correspond to maps $m_{k,i} : A \otimes \dots \otimes A \otimes M \otimes A \otimes \dots \otimes A \to M$ which endow $M$ with a $C_\infty$-module structure.
    \item The generators $f_{k,i} \coloneqq \mu^{\antishriek}_{\cU'[k, i]}$, where:
          \begin{equation*}
            \cU'[k, i] = \underbrace{\bigl( \emptyset, \dots, \emptyset, \overbrace{\emptyset \subset X}^{i\text{th position}}, \emptyset, \dots, \emptyset \bigr)}_{k \text{ terms}}.
          \end{equation*}
  \end{enumerate}
  These correspond exactly the the $C_\infty$-module structure maps and the $\infty$-$C_\infty$-module morphism maps.
  It is a simple exercise to check that the differential from Proposition~\ref{prop: hodisjexplicit} are identical to the ones from the definition of $C_\infty$-modules.
\end{proof}

The proof of the next lemma is just as straightforward, if tedious.

\begin{lemma}\label{lem sierp}
  Let $S$ be the Sierpiński space, with points $\{1,2\}$ and opens $\{\emptyset, \{1\}, \{1, 2\}\}$.
  A $\hoDisj$-algebra on $S$ is a septuple $\cA = (A, M, \eta_M, N, \eta_N, f, h)$ where:
  \begin{itemize}[nosep]
    \item $A = \cA(\emptyset)$ is a $C_\infty$-algebra;
    \item $M = \cA(\{1\})$ is a $C_\infty$-$A$-module;
    \item $\eta_M = \cA(\emptyset \subset \{1\})$ is a pointing $(\id_A, \eta_M) : (A,A) \leadsto (A,M)$;
    \item $N = \cA(\{1,2\})$ is a $C_\infty$-$A$-module;
    \item $\eta_N = \cA(\emptyset \subset \{1,2\})$ is a pointing $(\id_A, \eta_N) : (A,A) \leadsto (A,N)$;
    \item $f = \cA(\{1\} \subset \{1,2\})$ is an $\infty$-$C_\infty$-module morphism $(A,M) \leadsto (A,N)$;
    \item $h = \cA(\emptyset \subset \{1\} \subset \{1,2\})$ is an $\infty$-module $\infty$-homotopy $f \circ \eta_M \simeq \eta_N : A \leadsto N$.
  \end{itemize}
\end{lemma}

\begin{example}
  Given a $\hoDisj_S$-algebra $(A, M, \eta_M, N, \eta_N, f, h)$, the homotopy $h$ induces the next relation from which $f \circ \eta_M$ and $\eta_N$ are equal in cohomology.
  This relation is illustrated by the next picture, where solid black edges are colored by $A$, red dashed edges by $M$, and blue double-dashed edges by $N$.
  \begin{align*}
    \begin{forest}
      smalloperad
      [[,mod2,baseline[,mod[1]]]]
    \end{forest}
     &
    \xmapsto{\partial}
    \begin{forest}
      smalloperad
      [[,mod2,baseline[1]]]
    \end{forest}
    -
    \Biggl(
    \begin{forest}
        smalloperad
        [[,mod2,baseline[1,mod]]]
      \end{forest}
    \circ
    \begin{forest}
        smalloperad
        [[,mod,baseline[1]]]
      \end{forest}
    \Biggr);
     &
    \text{or algebraically: }
    (\partial h)(a)
     & =
    \eta_N(a) - f(\eta_M(a)),
  \end{align*}
  But there are of course other relations, such as:
  \begin{equation*}
    \begin{forest}
      smalloperad
      [[,mod2,baseline[,mod2[,mod[1]]][2]]]
    \end{forest}
    \xmapsto{\partial}
    \begin{forest}
      smalloperad
      [[,mod2,baseline[,mod2[1]][2]]]
    \end{forest}
    \pm
    \Biggl(
    \begin{forest}
        smalloperad
        [[,mod2,baseline[,mod2[1,mod]][2]]]
      \end{forest}
    \circ_1
    \begin{forest}
        smalloperad
        [[,mod,baseline[1]]]
      \end{forest}
    \Biggr)
    \pm
    \Biggl(
    \begin{forest}
        smalloperad
        [[,mod2,baseline[1,mod2][2]]]
      \end{forest}
    \circ_1
    \begin{forest}
        smalloperad
        [[,mod2,baseline[,mod[1]]]]
      \end{forest}
    \Biggr)
    \pm
    \Biggl(
    \begin{forest}
        smalloperad
        [[,mod2,baseline[1,mod]]]
      \end{forest}
    \circ
    \begin{forest}
        smalloperad
        [[,mod,baseline[,mod[1]][2]]]
      \end{forest}
    \Biggr)
    \pm
    \Biggl(
    \begin{forest}
        smalloperad
        [[,mod2,baseline[,mod[1]]]]
      \end{forest}
    \circ
    \begin{forest}
        smalloperad
        [[,baseline[1][2]]]
      \end{forest}
    \Biggr)
    .
  \end{equation*}
  Algebraically, this becomes (where we again avoid indices):
  \begin{equation*}
    (\partial h)(a, b) = \eta_N(a, b) \pm f(\eta_M(a), b) \pm \mu_N(h(a), b) \pm f(\eta_N(a, b)) \pm h(\mu_A(a, b)).
    \qedhere
  \end{equation*}
\end{example}

\subsection{The prefactorization algebra on \texorpdfstring{$\RR$}{R} associated to a dg Lie algebra}

Let $\fg$ be a Lie algebra.
\Textcite{CG1} discussed a factorization algebra $\widetilde \cF_{\fg}$ on $\RR$ whose cohomology factorization algebra is the factorization algebra $\cF_{\fg}$ associated to the universal enveloping algebra $U(\fg)$.
Explicitly,
\begin{equation*}
  \tilde \cF_{\fg}(U) = C_\bullet(\Omega_c(U)\otimes \fg),
\end{equation*}
where the Chevalley-Eilenberg chains $C_\bullet(\_)$ are computed using the bornological tensor product (see \autocite[Section~3.4]{CG1} for details).
In the aforementioned reference, it is shown that
\begin{equation*}
  H^\bullet(\widetilde \cF_\fg) \cong \cF_{\fg},
\end{equation*}
and, in particular, the cohomology is concentrated in cohomological degree 0.
Since $\tilde \cF_{\fg}$ is concentrated in non-positive cohomological degree, there is a natural quasi-isomorphism $\tilde \cF_\fg\to H^\bullet(\widetilde \cF_\fg)$.
A minor modification of these arguments applies also in the case that $\fg$ is a graded Lie algebra.
Our techniques make it straightforward to describe a map the other way, and in fact we can show that the theory applies equally well for dg Lie algebras:

\begin{lemma}
  Let $(\fg, d_\fg, [\_, \_])$ be a dg Lie algebra.
  There is an infinity-quasi-isomorphism of $\hoDisj$-algebras:
  \begin{equation*}
    \cF_\fg \rightsquigarrow \tilde \cF_{\fg}.
  \end{equation*}
\end{lemma}
\begin{proof}
  Let $U\subseteq \RR$ be open.
  Let $d_{dR}$ denote the de Rham differential on the dg Lie algebra $\Omega^\bullet_c(U)\otimes \fg$.
  The complex $\widetilde \cF_\fg(U)$ has a differential induced from $d_{dR}$, a differential induced from $d_\fg$, and finally the Chevalley-Eilenberg differential $d_{CE}$ (note that we use the term $d_{CE}$ to denote only the part of the Chevalley-Eilenberg differential which is not a derivation for the symmetric algebra structure).
  There is a deformation retraction
  \begin{equation*}
    \ourDR{(\bigoplus_{\pi_0(U)}\fg,d_\fg)}{(\Omega^\bullet_c(U)\otimes \fg[1], d_{dR}+d_\fg)}{i_U}{p_U}{h_U},
  \end{equation*}
  where $p_U$ is, connected-component-by-connected-component, the integration map.
  To define the maps $i_U$ and $h_U$, one makes a choice of compactly-supported one-form on $(-1/2,1/2)$ with integral 1.
  This is a fairly standard deformation retraction satisfying the side conditions $h_U^2=p_U h_U = h_U i_U = 0$; see, e.g.\ \autocite[Section~4.3]{botttu}.
  This deformation retraction extends to a retraction
  \begin{equation*}
    \ourDR{(\bigotimes_{\pi_0(U)}\Sym\fg,d_{\fg})}{(\Sym(\Omega^\bullet_c(U)\otimes \fg[1]),d_{dR}+d_\fg)}{i_U}{p_U}{h_U}
  \end{equation*}
  satisfying the same side conditions.
  The complex $\widetilde \cF_\fg$ is a perturbation of the right-hand side of the above retraction, namely by the differential $d_{CE}$.
  The homological perturbation lemma \autocite{crainic} applies in this situation: the differential $d_{CE}$ lowers sym-degree by 1 and $h_U$ preserves sym-degree, so the infinite sum $\sum_{n=0}^\infty (d_{CE}h_U)^n$ is well-defined on any symmetric tensor of fixed sym-degree.
  Hence, we obtain a deformation retraction
  \begin{equation*}
    \ourDR{(\bigotimes_{\pi_0(U)}\Sym\fg,d_\fg+\delta)}{\widetilde \cF_{\fg}}{i'_U}{p'_U}{h'_U},
  \end{equation*}
  where
  \begin{align*}
    i'_U = \sum_{n=0}^\infty (h_Ud_{CE})^n i_U, \quad p'_U & = \sum_{n=0}^\infty p_U(d_{CE}h_U)^n, \quad h'_U =  \sum_{n=0}^\infty h_U(d_{CE}h_U)^n \\
    \delta                                                 & = p_U \sum_{n=0}^\infty (\eta_U d_{CE})^n i_U.
  \end{align*}
  The differential $d_{CE}$ is trivial on the image of $i_U$; hence, $\delta=0$.
  For the same reason, we get $i'_U = i_U$.
  The homotopy transfer theorem (Theorem~\ref{prop HTT}) now guarantees the existence of a $\hoDisj$-algebra structure on the collection of graded vector spaces $\bigotimes_{\pi_0(U)}\Sym \fg$ (as $U$ ranges over the open subsets of $\RR$), and an infinity-morphism from this $\hoDisj$-algebra to $\widetilde{\cF}_\fg$.
  Note that, by the PBW theorem, $\bigotimes_{\pi_0(U)}\Sym(\fg)\cong\bigotimes_{\pi_0(U)}U(\fg)=\cF_\fg(U)$.

  Our goal now is to show that this transfered $\hoDisj$-algebra structure coincides with the $\Disj$-algebra structure on $\cF_\fg$ constructed from the associative product on~$U(\fg)$.
  Since the formulas for the transfered $\hoDisj$-algebra structure on $\widetilde{\cF}_\fg$ do not depend in any way on $d_\fg$, it suffices to consider the case that $d_\fg=0$.
  Let us proceed to show that there are no operations of negative degree in the  transferred $\hoDisj$-algebra structure on $\cF_\fg(U)$.

  To this end, it is useful to introduce a grading on $\widetilde \cF_{\fg}$ which we call the syzygy-degree, by analogy with the case of the bar and cobar constructions of a quadratic algebra/coalgebra.
  The syzygy degree in $\widetilde{\cF}_\fg(U)$ is the total sym-degree minus the form degree.
  The syzygy degree for all elements of $\cF_\fg(U)$ is zero.
  With these choices, the maps $i_U$, $p_U$, $h_U$, and $d_{CE}$ have syzygy degrees 0, 0, $+1$, and $-1$, respectively.
  It follows that $i'_U$, $p'_U$, and $h'_U$ have the same degrees as their un-primed counterparts.
  Moreover, the $\Disj$-operations on $\widetilde{\cF}_\fg$ have syzygy degree 0.

  The operations of the $\hoDisj$-algebra structure on $\cF_\fg$ induced from homotopy transfer can be represented by a sum of trees with vertices of valence two or three.
  The leaves of each tree are labeled by $i'_U$ (possibly for varying values of $U$), the internal edges are labeled by $h'_U$, and the root is labeled by $p'_U$.
  The vertices are labeled by operations from $\widetilde{\cF}_\fg$.
  The total syzygy degree of such a tree is just the number of internal edges; since $\cF_\fg$ is concentrated in syzygy degree 0, it follows that no trees with internal edges contribute to the $\hoDisj$-algebra structure on $\cF_\fg$.
  But the trees with no internal edges are precisely the ones that induce the $\Disj$-algebra structure which $\cF_\fg$ obtains as the cohomology of $\widetilde \cF_\fg$.
  This $\Disj$-algebra structure is precisely the one associated to the universal enveloping algebra $U(\fg)$~\autocite{CG1}.
\end{proof}

\printbibliography


\end{document}